\documentclass[a4paper,11pt]{article}
\usepackage[utf8]{inputenc} 
\usepackage[english]{babel}
\usepackage[T1]{fontenc}   
\usepackage[normalem]{ulem}
\usepackage{mathpazo}
\usepackage{graphicx}
\usepackage{comment}
\usepackage{amsthm,amsmath,amsfonts,stmaryrd,mathrsfs,amssymb,mathtools}
\usepackage{mathtools}
\usepackage{bbm}
\usepackage{enumitem}
\usepackage{chemarrow}
\usepackage{tikz}
\usepackage[top=2.5cm, bottom=2.5cm, left=2cm, right=2cm]{geometry}
\usepackage{subfig}
\usepackage{multirow}
\usepackage{physics}
\usepackage{array}
\usepackage{bigints}
\usepackage{etoolbox}
\usepackage{mdframed}
\usepackage{color}
\usepackage{booktabs}
\usepackage{textgreek}
\usepackage[unicode,colorlinks=true,linkcolor=blue,citecolor=red,pdfencoding=auto,psdextra]{hyperref}
 \usepackage{dsfont}
\usepackage{mdframed}
 
\usepackage{enumitem}
\usepackage{amsmath}
\usepackage{upgreek}
\usepackage{commath}
\usepackage{amssymb}
\usepackage{xcolor}
\usepackage{amsthm}
\usepackage{csquotes}
\usepackage{titlesec}
\makeatletter
\newcommand{\optionaldesc}[2]{%
  \phantomsection
  #1\protected@edef\@currentlabel{#1}\label{#2}%
}
\makeatother

\newtheorem{theorem}{Theorem}[section]
\newtheorem{lemma}[theorem]{Lemma}
\newtheorem{proposition}[theorem]{Proposition}
\newtheorem{corollary}[theorem]{Corollary}
\newtheorem{remark}[theorem]{Remark}
\newtheorem{definition}[theorem]{Definition}

\newcommand{\E}{\mathbb{E}}
\DeclareRobustCommand{\P}{\mathbb{P}}

\newcommand{\Var}{\mathrm{Var}}
\newcommand{\dTV}{d_{\mathrm{TV}}}
\newcommand{\T}{\mathcal{T}}
\newcommand{\Po}{\mathrm{Poisson}}
\newcommand{\Obig}{\mathcal{O}}

\newcommand{\N}{\mathbb{N}}
\newcommand{\Z}{\mathbb{Z}}
\newcommand{\R}{\mathbb{R}}
\newcommand\Es[1]{\mathbb{E}\left[#1\right]}
\renewcommand\Pr[1]{\mathbb{P}\left(#1\right)}

\title{Rare subtree patterns in size-conditioned Bienaymé trees: Poisson approximation and declumping}

\author{
Igor Kortchemski 
\thanks{CNRS \& DMA, École normale supérieure, PSL University, 75005 Paris, France, \textsf{igor.kortchemski@math.cnrs.fr}
} 
\qquad
Leonard Vetter 
\thanks{IOR, Karlsruhe Institute of Technology, Kaiserstr.
12, 76131 Karlsruhe, Germany, \textsf{leonard.vetter@kit.edu}
}   
}\date{}

\begin{document}
\maketitle

\begin{abstract}
We establish a general Poisson approximation for rare local patterns in
critical Bienaym\'e--Galton--Watson trees with offspring distribution $\mu$ in
the domain of attraction of a stable law, conditioned to have a large number of vertices. A
pattern is specified by a sequence-dependent mark on fringe subtrees. If
marked fringe subtrees remain microscopic and nearby marked occurrences have
negligible clustering, then their count is asymptotically Poisson in
total variation whenever its mean remains bounded; when the mean diverges, the
count satisfies a law of large numbers. The main difficulty is the global
dependence created by size conditioning. We overcome it by combining the
cyclic-shift representation with a refined form of the Chen--Stein bound and a bridge-removal estimate
controlling the interaction between a local mark and the remainder of the
conditioned random walk.

For non-fringe patterns, overlapping occurrences may form clusters and the raw
count need not be asymptotically Poisson. We introduce declumped indicators
which select boundary witnesses of these clusters and prove a general Poisson
approximation for their count. As applications, we obtain sharp asymptotics
for the maximum leaf-height, equivalently the maximum protection number, and for the height of the largest complete
$r$-ary tree appearing as a non-fringe subtree. Unary-chain maxima, and the
maximum leaf-height when $\mu_1>0$, exhibit lattice-modulated Gumbel
behavior. Complete $r$-ary patterns for $r\ge2$, and the maximum leaf-height when $\mu_1=0$, are localized on one or two consecutive integers. The
results require no exponential moment and include offspring distributions
with infinite variance.
\end{abstract}

%==================================================================
\section{Introduction}\label{sec:intro}
%==================================================================

Rare local patterns in a size-conditioned Bienaymé tree involve two distinct sources of dependence. Conditioning on the total size creates a global dependence between otherwise local configurations, while overlapping non-fringe occurrences may create strong local clustering. Either phenomenon can invalidate a direct Poisson approximation. The purpose of this paper is to separate these two obstructions and provide general criteria for overcoming them.

The local structure and fringe-subtree statistics of random trees have been
studied from several viewpoints, including fringe limits, additive
functionals, protection parameters, and counts of repeated or distinct fringe
subtrees
\cite{Ald91,BJ01,Jan12,DJ14,DW15,Jan16,Jan21,FS22,BW22,FSW24,RR26}.
Poisson approximation for rare fringe-subtree counts, which is closest to our
first main result, was developed in \cite{CD17,BHJ26}.

Our first contribution, Theorem \ref{thm:PoissonApproximation}, is a general
Poisson approximation for rare sequence-dependent marks on fringe subtrees.
The assumptions isolate two natural properties: marked fringe subtrees must
remain microscopic in the conditioned tree, and nearby marked occurrences
must not form significant clusters. The relation with \cite{CD17,BHJ26} is complementary rather than one of strict
inclusion. Cai--Devroye treat growing fringe families and complete non-fringe
patterns in the finite-variance setting. Berzunza--Holmgren--Janson treat
several structured growing fringe counts and also obtain explicit
finite-$n$ total-variation bounds and normal regimes. Our criterion is
formulated for general sequence-dependent marks under stable offspring laws
and is coupled with a separate declumping theorem for clustered non-fringe
patterns.

Our second contribution, Theorem \ref{thm:PoissonApproximation2}, is a
root-declumping principle for non-fringe patterns. A direct Poisson
approximation for their raw count may fail because a single configuration can
create a cluster of overlapping occurrences, as already observed for complete
unary non-fringe subtrees by Cai and Devroye \cite{CD17}. We instead count
boundary witnesses of these clusters and prove a Poisson approximation for the
resulting declumped count.

Finally, we apply the framework to the maximum leaf-height, equivalently the
maximum protection number, and to the height of the largest complete $r$-ary
tree appearing as a non-fringe subtree. Unary-chain maxima, and the maximum
leaf-height when $\mu_1>0$, exhibit lattice-modulated Gumbel behavior.
Complete $r$-ary patterns for $r\ge2$, and the maximum leaf-height when
$\mu_1=0$, are localized on one or two consecutive integers. In particular,
the extension beyond finite variance and exponential moments is a consequence
of a broader probabilistic framework rather than the sole purpose of the
paper; see \cite{DGZ23,HSW24} for the previous leaf-height results.

\paragraph{Setting.} We consider Bienaym\'e trees (also often called Galton--Watson trees) with a critical offspring distribution $\mu=(\mu_j)_{j\ge0}$, that is, $\sum_{j\ge0}j\mu_j=1$, and assume $\mu_1\neq1$. We further assume that $\mu$ belongs to the domain of attraction of an $\alpha\in(1,2]$ stable law. This means that there exists a slowly varying function $\ell$ such that $ \sum_{j=0}^{n} j^{2} \mu_{j} =\ell(n) n ^{2-\alpha}$, and there exists an increasing sequence $(a_{n})$ such that if $(K_{i})_{i \geq 1}$ are i.i.d. random variables with law $\mu$,
\begin{equation}
\label{eq:an} \frac{K_{1}+ \cdots+K_{n}-n}{a_{n}}  \quad \mathop{\longrightarrow}^{(d)}_{n \rightarrow \infty} \quad Y_{\alpha},
\end{equation}
where $Y_{\alpha}$ has Laplace exponent given by $\Es{ \exp(- \lambda Y_{\alpha})}= \exp(\lambda ^ \alpha)$ for every $\lambda>0$, see Sec.~\ref{sec:toolbox} for details.

Let $\mathcal T$ be a $\mu$-Bienaym\'e tree. For $n\ge1$, let $\mathcal T^n$ denote a $\mu$-Bienaym\'e tree conditioned to have $n$ vertices, restricting implicitly to admissible values of $n$ such that the probability that the size of $ \mathcal{T}$ is $n$ is positive.

\subsection{Poisson approximation for marked fringe subtrees}

Our main result is a general Poisson approximation for the number of vertices of
$\mathcal T^n$ whose fringe subtree is \emph{marked} by a sequence of indicator functions on plane trees.
This result is useful in its own right, as it unifies and extends Poisson approximation results obtained in different contexts in the literature.

Given a plane tree $T$ and a vertex $v\in T$, by definition the fringe
subtree $T_v$ consists of $v$ and all its descendants. Let
$(G_k)_{k\ge1}$ be a sequence of $\{0,1\}$-valued functions on finite plane
trees, and fix a sequence $(k_n)$ of integers. We are interested in the asymptotic behavior of the quantity
$$
\#\{u\in\mathcal T^n : G_{k_n}(\mathcal T^n_u)=1\},
$$
which is the number of marked fringe subtrees of  $\mathcal T^n$ (if  $G_{k_n}(\mathcal T^n_u)=1$ we say that the fringe subtree  $\mathcal T^n_u$ is marked).

We first state the assumptions used in the limit theorems. Throughout, all the assumptions are stated with respect to a \emph{cutoff
sequence} $(M_n)$ of positive integers which satisfies 
\begin{center}
\fbox{
$M_n\to\infty$ and
$M_n=o(n)$.}
\end{center}
We write $|T|$ for the number of vertices of a plane tree $T$.
Whenever a conditional law given a marked event is used, we work along subsequences on which the corresponding marking probability is positive.

\paragraph{Main assumptions.} Although three assumptions are introduced below, their logical roles are simple. Assumption \ref{hyp:S} is a convenient structural criterion implying the microscopicity Assumption \ref{hyp:M}; Assumption \ref{hyp:C} is the non-clustering condition used for Poisson approximation.

\paragraph{Microscopicity assumptions.}
We first introduce conditions ensuring that marked fringe subtrees remain
microscopic in the conditioned tree.
\begin{mdframed}
\textbf{Assumption \textcolor{blue}{\optionaldesc{(M)}{hyp:M}} for $(G_k)$
along $(k_{n})$ with respect to $(M_{n})$: microscopicity.}
Let $U_n$ be a uniform vertex of $\mathcal T^n$, independent of
$\mathcal T^n$. Then, as $n \rightarrow \infty$,
$$
\P\bigl(G_{k_n}(\mathcal T^n_{U_n})=1\bigr)\sim \P(G_{k_n}(\mathcal T)=1)
\qquad \textrm{and} \qquad
\P\bigl(\ |\mathcal T^n_{U_n}|> M_n \mid G_{k_n}(\mathcal T^n_{U_n})=1\bigr)
=o(1).
$$
\end{mdframed}

The first asymptotic in \ref{hyp:M} is equivalent to saying that the
expected proportion of marked fringe subtrees of $\mathcal{T}^{n}$ is asymptotic to $\P(G_{k_n}(\mathcal T)=1)$ as
$n \rightarrow\infty$. The second condition says that,
among marked fringe subtrees seen from a uniform vertex of $\mathcal T^n$,
those whose size exceeds the cutoff $M_n$ are negligible. A standard choice of the cutoff sequence is $M_n = \lfloor n^{1-\eta} \rfloor$ for some $\eta \in (0,1)$ small enough.  See also
Lemma \ref{lem:M} for another simple criterion implying Assumption
\ref{hyp:M}.

The following stronger structural assumption is sometimes simpler to check in applications, and will also be used to verify Assumption \ref{hyp:M} for  declumped indicators. Recall that the sequence $(a_{n})$ is defined by \eqref{eq:an} (see also Sec.~\ref{sec:toolbox}). 

\begin{mdframed}
  \textbf{Assumption \textcolor{blue}{\optionaldesc{(S)}{hyp:S}} for
  $(G_k)$ along $(k_n)$ with respect to $(M_{n})$: structural microscopicity.} Under the conditional law
  $$
    \P_k(\,\cdot\,)\coloneqq\P(\,\cdot\mid G_k(\mathcal T)=1),
    \qquad
    \E_k[\,\cdot\,]\coloneqq\E[\,\cdot\mid G_k(\mathcal T)=1],
  $$
there exist nonnegative integer-valued random variables $C_k,L_k$ and
random trees $\mathcal T^{(1)},\mathcal T^{(2)},\ldots$ such that,
conditionally on $(C_k,L_k)$, the trees $\mathcal T^{(i)}$ are independent and
each has the same law as $\mathcal T$, and
$$
    |\mathcal T|
    =
    C_k+\sum_{i=1}^{L_k}|\mathcal T^{(i)}|
    \quad\text{under } \P_k$$
    and for every $\varepsilon>0$ we have $
a_n\P_{k_n}(C_{k_n}>\varepsilon M_n) \to 0$ and ${\E_{k_n}[L_{k_n}]}/{a_{M_n}} \to 0$.
\end{mdframed}

Assumption \ref{hyp:S} says that, conditionally on the event
$G_k(\mathcal T)=1$, the marked event can be witnessed by a small skeleton
with few open leaves. Hence marked fringe subtrees are unlikely to be large.

In practice, it is sometimes simpler to check the condition  $\E_{k_n}[C_{k_n}+L_{k_n}]
    =
    n^{o(1)}$ which does not involve $(a_{n})$, which then implies that Assumption \ref{hyp:S} holds for
  $(G_k)$ along $(k_n)$ with respect to $(M_{n})$  with a cutoff of the form $M_n=\lfloor n^{1-\eta}\rfloor$ for every $\eta\in(0,1-1/\alpha)$ (Lemma \ref{lem:critM}).  We will also later  show that Assumption \ref{hyp:S} implies Assumption
\ref{hyp:M} (Corollary \ref{cor:SM}).

\paragraph{Non-clustering assumption.}
We next introduce a condition ensuring that local clustering of marked
occurrences is negligible.
For a finite tree $T$, let $u_j(T)$ denote its $j$-th vertex in depth-first
order, whenever it exists. For $\ell\ge1$, let $A_n(\ell)$ be the
probability that the fringe subtrees at the root of $\mathcal T$ and at its
$(\ell+1)$-th vertex in depth-first order are both present and both marked:
$$
A_n(\ell)
=
\P\bigl(
|\mathcal T|\ge \ell+1,\
G_{k_n}(\mathcal T)=1,\
G_{k_n}(\mathcal T_{u_{\ell+1}(\mathcal T)})=1
\bigr).
$$
Assumption \ref{hyp:C} quantifies how strongly marked
fringe subtrees clump together.

\begin{mdframed}
  \textbf{Assumption \textcolor{blue}{\optionaldesc{(C)}{hyp:C}} for
  $(G_{k})$ along $(k_{n})$ with respect to $(M_{n})$: pair non-clustering.} We have
  $$ \sum_{\ell=1}^{M_n} A_n(\ell)=o\bigl(n\,\P(G_{k_n}(\mathcal T)=1)^2\bigr). $$
\end{mdframed}

A practical way to check this assumption is to show that $
\sup_{\ell\ge1} A_n(\ell)
=
 o \big( \frac{n}{M_{n}}\, \P(G_{k_n}(\mathcal T)=1)^2\big)$.

\paragraph{Main result.} We are now ready to state our main Poisson approximation result. 

\begin{theorem}
\label{thm:PoissonApproximation}
Assume that $\mu$ is critical and belongs to the domain of
attraction of an $\alpha$-stable law for some $\alpha\in(1,2]$.
Set $\pi_n=\P(G_{k_n}(\mathcal T)=1)$. Assume that Assumptions \ref{hyp:M} and \ref{hyp:C} hold for
$(G_k)$ along $(k_n)$ with respect to some cutoff sequence $(M_n)$.
\begin{enumerate}
\item[(i)]  If
$(n\pi_n)$ is bounded, then
\begin{equation}
\label{eq:DTVthm}
\dTV \left( \# \{u \in \mathcal{T}^{n} : G_{k_{n}}( \mathcal{T}^{n}_{u})=1\},\Po(n \pi_{n})\right)  \quad \mathop{\longrightarrow}_{n \rightarrow \infty} \quad 0.
\end{equation}
\item[(ii)]  If
$n\pi_n\to\infty$, then
$$
\frac{
\#\{u\in\mathcal T^n: G_{k_n}(\mathcal T^n_u)=1\}
}{n\pi_n}  \quad \mathop{\longrightarrow}^{(\P)}_{n \rightarrow \infty} \quad 1.$$
In particular,
$
\#\{u\in\mathcal T^n: G_{k_n}(\mathcal T^n_u)=1\} \rightarrow \infty$ in
probability.
\end{enumerate}
\end{theorem}

Theorem \ref{thm:PoissonApproximation} is our basic Poisson approximation
result for rare marked fringe subtrees. It can be read as a general-purpose
criterion: once microscopicity and a suitable non-clustering condition have
been verified, the marked count is asymptotically Poisson in the bounded-mean
regime, and satisfies a law of large numbers in the divergent-mean regime.

Theorem \ref{thm:PoissonApproximation} recovers and extends several known
Poisson approximations for fringe subtree counts. We illustrate this in
Section \ref{sec:applications} through three direct applications. First, when
the marking function forces the size of the marked fringe subtree to be
$k_n$, the local clustering conditions become automatic. This gives
Corollary \ref{cor:Poisson_fringe}, which recovers the fringe-subtree results
of Cai--Devroye~\cite{CD17} and the degree-statistic version of
Berzunza--Holmgren--Janson~\cite{BHJ26}. Second, using Assumption
\ref{hyp:S} to verify microscopicity, Theorem
\ref{thm:PoissonApproximation} applies to raw non-fringe subtree counts in
regimes where overlapping occurrences do not create significant local
clusters. Third, the same approach gives Poisson approximations and a law of
large numbers for vertices with unusually large prescribed outdegree.

These applications should be contrasted with the extremal non-fringe problems
treated later in the paper. For the maximum leaf-height and for complete
unary subtrees, raw occurrences typically form clusters, and a direct Poisson
approximation for the raw count is no longer the right object. For these
problems we introduce declumped indicators in the next subsection and use
Theorem \ref{thm:PoissonApproximation2} instead.

\paragraph{A convenient sufficient criterion.}
Suppose that $\E_{k_n}[C_{k_n}+L_{k_n}]=n^{o(1)}$ and $
  \sup_{\ell\ge1}A_n(\ell)\le n^{o(1)}\pi_n^2$.
Then, for every fixed $\eta\in(0,1-1/\alpha)$, the assumptions of Theorem
\ref{thm:PoissonApproximation} hold with
$M_n=\lfloor n^{1-\eta}\rfloor$.

\subsection{Declumping}

Our main applications involve extremal statistics which are naturally expressed
in terms of non-fringe subtrees.  A direct Poisson approximation for raw non-fringe counts is in general false.
Indeed, raw non-fringe occurrences may form clusters. For instance, Cai and
Devroye~\cite[Lemma~5.5]{CD17} observed that, when
$n\mu_1^{k_n}\to\lambda>0$, the number of complete unary non-fringe subtrees
of height $k_n$ in $\mathcal T^n$ is not asymptotically Poisson with parameter
$\lambda$. Intuitively, a long unary chain creates many overlapping unary
chains of nearby heights. This is analogous to the clustering of long runs in
sequences of independent Bernoulli random variables.

To handle this phenomenon, we introduce declumped indicators. The idea is to
mark only boundary points of clusters. More precisely, let $(G_k)_{k\ge1}$ be
a sequence of $\{0,1\}$-valued functions on finite plane trees, and let
$B\subset\mathbb Z_+$. We define
$$
  \widehat{G}_k(T)
  =
  \mathbbm{1}_{\{k_\varnothing(T)\in B\}}
  \mathbbm{1}_{\{\exists u\in\Gamma_\varnothing(T):G_k(T_u)=1\}},
$$
where $\Gamma_\varnothing(T)$ denotes the set of children of the root of $T$.
Thus $\widehat G_k(T)=1$ means that the root of $T$ satisfies a degree
condition and has at least one child whose fringe subtree is marked by $G_k$.
In the applications, the set $B$ is chosen so that such vertices correspond to
the boundary of a cluster of raw occurrences. See Figure \ref{fig:intro} for an illustration of the declumped indicators, and observe that we count two vertices that have $H_{1,3}$ as non-fringe subtree, while the declumped count is only one.

\begin{theorem}
\label{thm:PoissonApproximation2}
Assume that $\mu$ is critical and belongs to the domain of
attraction of an $\alpha$-stable law for some $\alpha\in(1,2]$. Let $(G_k)$
be a sequence of $\{0,1\}$-valued functions on finite plane trees,  let
$(k_n)$ be a sequence of integers and let $(M_n)$ be a cutoff sequence. Set
$$
  \pi_n=\P(G_{k_n}(\mathcal T)=1).
$$
Let $B\subset\mathbb Z_+$ be such that
$$
  c_B
  =
  \E\left[
    k_\varnothing(\mathcal T)
    \mathbbm{1}_{\{k_\varnothing(\mathcal T)\in B\}}
  \right]
  \in(0,\infty).
$$
Define
$$
  \widehat{G}_k(T)
  =
  \mathbbm{1}_{\{k_\varnothing(T)\in B\}}
  \mathbbm{1}_{\{\exists u\in\Gamma_\varnothing(T):G_k(T_u)=1\}}.
$$
Assume that $(G_k)$ satisfies Assumption
\ref{hyp:S} along $(k_n)$ with respect to $(M_n)$, and that
$(\widehat G_k)$ satisfies Assumption \ref{hyp:C} along $(k_n)$ with respect
to the same cutoff sequence $(M_n)$.
\begin{enumerate}
\item[(i)] If $(n\pi_n)$ is bounded, then
$
  \dTV \big(
    \#\{u\in\mathcal T^n:\widehat{G}_{k_n}(\mathcal T^n_u)=1\},
    \Po(c_B n\pi_n)
  \big)
  \to 0$.
\item[(ii)] If $n\pi_n\to\infty$, then
$
{
\#\{u\in\mathcal T^n:
\widehat G_{k_n}(\mathcal T^n_u)=1\}
}
\to \infty$ in probability.
\end{enumerate}
\end{theorem}

\begin{figure}[h]
    \centering
    \includegraphics[width=0.6\linewidth]{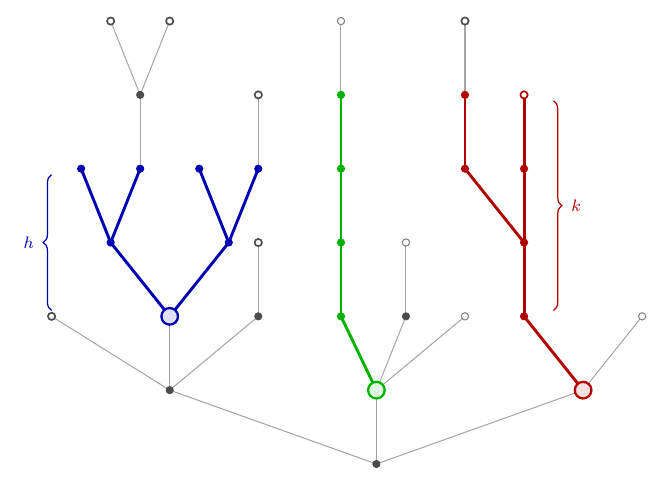}
\caption{Non-fringe subtrees and leaf-height. The fringe subtree at the blue
vertex supports a non-fringe copy of the complete $2$-ary tree of height $2$
(that is, $H_{2,2}\prec\mathcal T_v$). The green vertex has more than one child,
at least one of which supports a non-fringe copy of the complete $1$-ary tree of
height $4$ (a unary chain, $H_{1,4}$). The red vertex has more than one child, at
least one of which has leaf-height $3$.} 
\label{fig:intro}
\end{figure}

\paragraph{Application to largest complete $r$-ary subtrees.}

Our first application concerns the height of the largest complete $r$-ary tree
appearing as a non-fringe subtree of $\mathcal T^n$, motivated by results of Cai \& Devroye \cite{CD17}, who
studied the value $h_{n}$ such that with high probability $ \mathcal{T}^{n}$ contains  all  complete $r$-ary trees of height at most $h_{n}$ as fringe subtrees, as well as the height of the maximal complete $r$-ary non-fringe subtree in $ \mathcal{T}^{n}$.

More precisely, for a
finite plane tree $T$, set
$$
H_r(T)
=
\max \{k \geq 0 : \exists u \in T : H_{r,k}\prec T_u\},
$$
where $H_{r,k}$ denotes the complete $r$-ary tree of height $k$ (with all leaves at graph distance $k$ from the root), and where
$T'\prec T$ means that $T$ can be obtained from $T'$ by grafting plane trees
onto the leaves of $T'$. See Figure \ref{fig:intro} for an illustration.

\begin{theorem}
\label{thm:max_string}Assume that $\mu$ is critical and belongs to the domain of attraction of a stable law.
\begin{enumerate}
\item[(i)] Assume that $\mu_{1}>0$.  Set
$$t_n^* = \frac{1}{\log(1/\mu_{1})} \log n + \frac{1}{\log(1/\mu_{1})} \log(1-\mu_1).$$
For every $t \in \mathbb{R}$ we have 
$$\P\bigl(H_{1} ( \mathcal{T}^{n})  \leq \lfloor t_n^* + t\rfloor\bigr) - \exp\bigl(-\mu_1^{t+1 - \{t_n^* + t\}}\bigr)   \quad \mathop{\longrightarrow}_{n \rightarrow \infty} \quad  0$$
\item[(ii)] Let $r\ge2$ be such that $\mu_r>0$, and set
$$
  t_n^{**}=\log_r\log n-
  \log_r\left(\frac{\log(1/\mu_r)}{r-1}\right).
$$
Write $m_n=\lfloor t_n^{**}\rfloor$ and
$\theta_n=\{t_n^{**}\}$, where $\{x\} = x - \lfloor x \rfloor$ is the fractional part, and define
$$
  J_n=
  \begin{cases}
    \{m_n-1,m_n\},&\theta_n\le\tfrac12,\\
    \{m_n,m_n+1\},&\theta_n>\tfrac12.
  \end{cases}
$$
Then $
  \P(H_r(\mathcal T^n)\in J_n)\longrightarrow1$.
Moreover, if $\liminf_n\min(\theta_n,1-\theta_n)>0$, then
$\P(H_r(\mathcal T^n)=m_n)\to1$.
\end{enumerate}
\end{theorem}

The proof illustrates the role of declumping in a particularly simple way. We
take
$$
G_k(T)=\mathbbm{1}_{\{H_{r,k}\prec T\}}.
$$
For $r=1$, raw unary occurrences form clusters, and the relevant declumped
indicator is
$$
\widehat{G}_{k}(T)
=
\mathbbm{1}_{\{k_{\varnothing}(T) \neq 1\}}
\mathbbm{1}_{\{\exists\, u \in \Gamma_{\varnothing}(T) \colon H_{1,k}\prec T_u\}}.
$$
For $r\ge2$, the raw complete $r$-ary count is already sufficiently controlled
in the divergent regime, and the concentration statement follows from the
law-of-large-numbers part of Theorem \ref{thm:PoissonApproximation}.

\paragraph{Application to the leaf-height.}

Our second application concerns the maximum leaf-height. Given a plane tree
$T$, the leaf-height $\lambda_v(T)$ of a vertex $v\in T$, sometimes also called its protection number, is the graph distance, in number of edges, from $v$ to its closest leaf descendant; in particular, a leaf has leaf-height $0$. The quantity
$$
\lambda(T)=\max_{v\in T}\lambda_v(T)
$$
is the leaf-height of $T$. See Figure \ref{fig:intro} for an illustration. The study of protection numbers in trees began with Cheon and Shapiro \cite{CS08}, and has been carried out for several models of random trees \cite{Man11,DP12,MW15,HJ15,GGLS23} including Bienaymé trees \cite{DJ14,Cop17,HP17,GGLS23,DGZ23,HSW24}. Our second application concerns limit theorems for the leaf-height of large size-conditioned Bienaymé trees, extending results obtained by Devroye--Goh--Zhao
\cite{DGZ23} and \cite{HSW24}.

\begin{theorem}
\label{thm:max_lh} Assume that $\mu$ is critical and belongs to the domain of attraction of a stable law.
\begin{enumerate}
\item[(i)]  Assume that $\mu_{1}>0$. There exists a constant $C_{\mu}>0$ such that the following holds. Set
$$k^*_{n} =  \frac{1}{\log(1/\mu_{1})}\log n +  \frac{1}{\log(1/\mu_{1})} \log\bigl((1-\mu_1)C_{\mu}\bigr).$$
Then for every $t \in \mathbb{R}$ we have
\[
\P\bigl(\lambda(\mathcal{T}^n) \leq \lfloor k^*_{n} + t\rfloor\bigr) - \exp\bigl(-\mu_1^{t+1 - \{k^*_{n} + t\}}\bigr) \to 0.
\]
\item[(ii)] Assume that $\mu_{1}=0$. Set $\kappa = \min\{i \geq 2 : \mu_i > 0\}$. There exists a constant $D_{\mu} \in (0,1)$ such that the following holds. Set
$$
k_{n}^{**} \coloneqq  \log_\kappa \log n - \log_\kappa \log(1/D_{\mu}).
$$
 Write $m_n=\lfloor k_n^{**}\rfloor$ and
$\theta_n=\{k_n^{**}\}$, and define
$$
  J_n=
  \begin{cases}
    \{m_n-1,m_n\},&\theta_n\le\tfrac12,\\
    \{m_n,m_n+1\},&\theta_n>\tfrac12.
  \end{cases}
$$
Then $\P(\lambda(\mathcal T^n)\in J_n)\longrightarrow1$. In addition, if $\liminf_n\min(\theta_n,1-\theta_n)>0$, then
$\P(\lambda(\mathcal T^n)=m_n)\to1$.
\end{enumerate}
\end{theorem}

The constants $C_{\mu}$ and $D_{\mu}$ come from the asymptotic behavior of the
root leaf-height in the unconditioned Bienaym\'e tree (see Lemma
\ref{lem:ell_asymp}):
$$C_\mu
=
\lim_{k\to\infty}\frac{\P(\lambda_\varnothing( \mathcal{T}) \geq k)}{\mu_1^k} \textrm{ when } \mu_1>0, \qquad
D_\mu
=
\lim_{k\to\infty}\P(\lambda_\varnothing( \mathcal{T}) \geq k)^{1/\kappa^k}  \textrm{ when } \mu_1=0,
$$
where $\varnothing$ is the root of $\mathcal{T}$.

In particular, Theorem \ref{thm:max_lh} implies that
$$
\textrm{when } \mu_{1}>0, \quad  \frac{\lambda(\mathcal T^n)}{\log n}
 \quad \mathop{\longrightarrow}^{(\P)}_{n \rightarrow \infty} \quad 
\frac{1}{\log(1/\mu_1)}; \qquad 
\textrm{when } \mu_{1}=0,   \quad \frac{\lambda(\mathcal T^n)}{\log_\kappa\log n}
 \quad \mathop{\longrightarrow}^{(\P)}_{n \rightarrow \infty} \quad 1.
$$
The first-order limits were obtained by Devroye--Goh--Zhao \cite{DGZ23} under
finite offspring variance. Heuberger--Selkirk--Wagner \cite{HSW24} obtained the
refined lattice behavior under a finite-exponential-moment assumption. Theorem
\ref{thm:max_lh} removes the latter assumption and covers stable offspring laws
with infinite variance.

For the proof  we take
$G_k(T)=\mathbbm{1}_{\{\lambda_\varnothing(T)\ge k\}}$.
Recalling that $\kappa=\min\{j\ge1:\mu_j>0\}$, when $1-\kappa\mu_\kappa>0$, the corresponding declumped indicator is
$$\widehat G_k(T)
=
\mathbbm{1}_{\{k_\varnothing(T)\ge\kappa+1\}}
\mathbbm{1}_{\{\exists u\in\Gamma_\varnothing(T):G_k(T_u)=1\}}.
$$
The declumped count gives lower-bound witnesses for large leaf-height, while
the raw count gives the matching first-moment upper bound.

\paragraph{Strategy of the proofs.}
Theorem \ref{thm:PoissonApproximation} is obtained by combining the
cyclic-shift representation of conditioned Bienaym\'e trees with the
Chen--Stein method. The cyclic-shift representation allows us to replace a
uniformly chosen fringe subtree of $\mathcal T^n$ by the first tree encoded by
a random walk bridge conditioned on $S_n=-1$. Thus the relevant count can be
studied as a sum of local indicators along a cyclically exchangeable sequence
of increments.

The proof then proceeds by truncating the indicators at the cutoff scale
$M_n$. Assumption \ref{hyp:M} ensures that this truncation does not change the
count asymptotically: marked fringe subtrees with size larger than $M_n$ have
negligible contribution. After truncation, each indicator only depends on a
block of $M_n=o(n)$ consecutive increments. This makes it possible to apply the
Chen--Stein method with dependency neighborhoods of size of order $M_n$.

There are three error terms to control. The first one is a first-order
local-dependence term and is small because $M_n=o(n)$ in the bounded-mean
regime. The second one measures the probability of two nearby marked
occurrences and is precisely where Assumption \ref{hyp:C} is used. The third
term is the most delicate one. It measures the interaction between a local
indicator and the rest of the configuration outside its dependency
neighborhood under the bridge conditioning. Rather than dominate this term by a mixing coefficient, we retain the
conditional test-function term in the Arratia--Goldstein--Gordon bound and
estimate it directly under the bridge conditioning. Uniform local limit estimates show that deleting a microscopic block only
creates an asymptotically negligible perturbation of the bridge endpoint.
This is the bridge-removal step. This
yields the Poisson approximation when $(n\pi_n)$ is bounded.

When $n\pi_n\to\infty$, we no longer need a Poisson approximation. Instead, we
prove a law of large numbers for the marked count. The expectation is
asymptotic to $n\pi_n$ by Assumption \ref{hyp:M}, while the variance is
controlled by splitting pairs of vertices into local and separated pairs. The
local contribution is controlled by Assumption \ref{hyp:C}, and separated
pairs are asymptotically independent after the same bridge-removal estimates.
This gives convergence of the count divided by $n\pi_n$ to $1$ in probability.

Theorem \ref{thm:PoissonApproximation2} is a consequence of Theorem
\ref{thm:PoissonApproximation}. The root-declumping estimate gives
$\P(\widehat G_{k_n}(\mathcal T)=1)\sim c_B\P(G_{k_n}(\mathcal T)=1)$ in the
bounded regime.  Assumption \ref{hyp:S} for $(G_k)$, together with Assumption \ref{hyp:C}
for $(\widehat G_k)$ with respect to the same cutoff sequence, implies the
microscopicity and non-clustering conditions needed to apply the general
Poisson approximation theorem to the declumped indicators.

The declumped count is not, in general, deterministically equivalent to the
raw non-fringe count. Instead, it provides a lower-bound witness for the
existence of a cluster, while the corresponding raw count provides first-moment
upper bounds. This combination is sufficient for the extremal applications
below: the complete $r$-ary subtree height and the maximum leaf-height.

\paragraph{Relation with earlier Chen--Stein arguments.} Applications of the Chen--Stein method to fringe-subtree counts in conditioned
Bienaym\'e trees are relatively recent. The preliminary arXiv version of
Cai--Devroye \cite{CD16} uses a fringe-subtree switching coupling together
with Stein's method for exchangeable pairs; this argument was replaced in the
published version \cite{CD17} by a conditional-binomial approach. More
recently, Berzunza--Holmgren--Janson \cite{BHJ26} obtained quantitative
Poisson approximations for prescribed fringe trees in random trees with a
given degree sequence using a conditional, Palm-type coupling, and also
revisited the Cai--Devroye exchangeable-pair construction.

\paragraph{Further directions.} Natural extensions include central limit theorems in the regime
$n\pi_n\to\infty$ and analogous results for distance-to-boundary statistics
recently considered in \cite{MS26}. See also
Section \ref{ssec:lhperspectives} for possible extensions to counts of vertices
with prescribed leaf-height. It would be interesting to investigate whether the declumping mechanism developed here extends to clustered rare events in other models.

\tableofcontents

%==================================================================
\section{Bienaymé trees, coding by walks and fringe subtree counts}\label{sec:BGW}
%==================================================================

\subsection{Plane trees and Bienaym\'e trees}
\label{sec:bgw_trees}

We consider plane trees, also known as rooted ordered trees (see Le~Gall~\cite[Sec.~1]{LG05} for background on plane trees and Bienaymé trees). For every plane tree~$T$ and every vertex $v \in T$, we denote by $k_v(T)$ the outdegree of~$v$, also called the number of children, and denote by $\Gamma_{u}(T)$ the children of $u$. Recall from the Introduction the definition of a fringe subtree and a non-fringe subtree. If $T'$ can be obtained from $T_\varnothing$ by replacing some (possibly none) of the fringe subtrees of $T_\varnothing$ with leaves (or, equivalently, if $T$ can be obtained from $T'$ by grafting trees onto the leaves of $T'$), we write $T' \prec T$. If $u_{1}, \ldots, u_{n}$ are the vertices of a finite plane tree $T$ with $n$ vertices listed in depth-first-search (DFS) order (sometimes also called lexicographical order), we define  $\mathbf{x}(T)=(k_{u_{1}}(T)-1, \ldots,  k_{u_{n}}(T)-1)$ and say that $\mathbf{x}(T)$ is the DFS sequence of $T$. 

Given a sequence $\mathbf{x}=(x_{i})_{i \geq 1}$ of integers, we set
$$ \zeta(\mathbf{x})= \inf \{k \geq 1 : x_{1}+\cdots+x_{k}=-1\} \in \N \cup \{+\infty\}$$
with the convention $\inf \varnothing=+\infty$.
If $\mathbf{x}$ is a finite sequence, $\zeta(\mathbf{x})$ is defined in the same way.

It is a simple matter to see that for every tree $T$ with $n$ vertices  $u_{1}, \ldots, u_{n}$ listed in DFS order, for every $1 \leq i \leq n$, if $ \mathbf{x}^{(i)}(T)=(k_{u_{i}}(T)-1, k_{u_{i+1}}(T)-1, \ldots,k_{u_{n}}(T)-1,k_{u_{1}}(T)-1, \ldots,k_{u_{i-1}}(T)-1)$ is the $i$-th cyclic shift of $\mathbf{x}(T)$, then 
\begin{equation}
\label{eq:fringe}
\left(\mathbf{x}^{(i)}_{1}(T), \ldots,\mathbf{x}^{(i)}_{\zeta(\mathbf{x}^{(i)}(T))}(T)\right)=\mathbf{x}(T_{u_{i}}).
\end{equation}
In words, $\mathbf{x}^{(i)}(T)$ up to time $\zeta(\mathbf{x}^{(i)}(T))$  is precisely the DFS sequence of the fringe subtree of $u_{i}$ in $T$.

Given a probability distribution $\mu = (\mu_n)_{n \geq 0}$ on~$\mathbb{Z}_{\geq 0}$, we denote by $\P_{\mu}$ the law of a Bienaym\'e tree with offspring distribution~$\mu$. For every finite plane tree~$T$, we have
\[
\P_\mu(T) = \prod_{u \in T} \mu_{k_u(T)}.
\]
We implicitly always assume that $\mu_1 \neq 1$ to avoid degenerate cases. As mentioned in the Introduction, we denote by $ \mathcal{T}$ a $\mu$-Bienaymé tree and by $ \T^{n}$ a $\mu$-Bienaymé tree conditioned on having $n$ vertices (we implicitly restrict to those values of $n$ for which this conditioning makes sense). We call such values of $n$ admissible, meaning $\P(|\mathcal T|=n)>0$. We say that $\mu$ is aperiodic if its support has span one, equivalently $\gcd\{i-j:\mu_i\mu_j>0\}=1$.
 
 We will always consider offspring distributions satisfying the following assumption for some $\alpha \in (1,2]$:
\begin{equation*}
  \mu \textrm{ is critical and in the domain of attraction of an }
  \alpha\textrm{-stable law}.
  \tag{$H_\alpha$}
  \label{eq:Halpha}
\end{equation*}
As mentioned in the Introduction, the second property is equivalent to the fact that there exists a slowly varying function $\ell$ such that $ \sum_{j=0}^{n} j^{2} \mu_j =\ell(n) n ^{2-\alpha}$ for every $n \geq 1$.

\subsection{Random walks and the \L ukasiewicz path}\label{sec:rw}
\label{ssec:RW}

A central tool for studying Bienaym\'e trees is their coding by random walks (see \cite[Sec.~6.1]{Pit06} for background and proofs for the results mentioned here).

Let $\mathbf{X}=(X_i)_{i \geq 1}$ be a sequence of i.i.d. random variables with law given by $\P(X_{1}=i)=\mu_{i+1}$ for $i \geq -1$. Set
$$S_0 = 0, \qquad S_k = \sum_{i=1}^k X_i, \qquad k \geq 1.$$

It is  well known  that $\mathbf{x}( \mathcal{T}^{n}) $, the DFS sequence of $\T^{n}$, has the same law as $(X_{1}, \ldots,X_{n}) $ under $ \P(\cdot \mid \zeta(\mathbf{X})=n) $ and that 
 if $I$ is a uniform random variable on $ \{1,2, \ldots,n\}$ independent of $ \mathcal{T}^{n}$, 
\begin{equation}
\label{eq:shifts}
\mathbf{x}^{(I)} ( \mathcal{T}^{n}) \quad  \mathop{=}^{(d)} \quad  (X_{1}, \ldots,X_{n}) \textrm{ under } \P(\cdot \mid S_{n}=-1).
\end{equation}
where we recall that $\mathbf{x}^{(I)} ( \mathcal{T}^{n})$ is the $I$-th cyclic shift of $\mathbf{x}( \mathcal{T}^{n}) $.

\subsection{Subtree counts using cyclic shifts}
\label{sec:indicators}

Here we explain how we can gain access to the leaf-height and the maximal complete $r$-ary tree that appears as a non-fringe subtree using cyclic shifts.

First, if $\mathbf{x}$ is a (finite or infinite) sequence of elements of  $\mathbb{Z}_{\geq -1}$ such that $\zeta(\mathbf{x})<\infty$, we let $T(\mathbf{x})$ be the tree whose DFS sequence is $(\mathbf{x}_{1}, \ldots,\mathbf{x}_{\zeta(\mathbf{x})})$.

\begin{definition}[Indicators]
\label{def:indicators}
Let $G$ be a $ \{0,1\}$-valued function defined on the set of all finite plane trees. 
For every  sequence $\mathbf{x}=(x_1, \ldots, x_n) \in \mathbb{Z}_{\geq -1} ^{n}$ and for every $M \in \N \cup \{\infty\}$ we define $I_1^{(M)}$ as follows:
\begin{enumerate}
\item[--] if $\zeta(\mathbf{x})=\infty$ or $\zeta(\mathbf{x})>M$, set $I_1^{(M)}(\mathbf{x}) = 0$;
\item[--] if  $\zeta(\mathbf{x}) \leq M$ set $I_1^{(M)}(\mathbf{x})=G(T(\mathbf{x}_{1}, \ldots,\mathbf{x}_{\zeta(\mathbf{x})}))$.
\end{enumerate}
Then, for every $1 \leq i \leq  n$, define $I_i^{(M)}(\mathbf{x}) =I^{(M)}_{1}(x_i, x_{i+1}, \ldots, x_n, x_1, \ldots, x_{i-1})$.
\end{definition}
 Observe that the dependence of $I_i^{(M)}(\mathbf{x})$ on $G$ is implicit. Also observe that given $G$, the value of $I_1^{(M)}(\mathbf{x})$ only depends on the values of $x_{1},  \ldots,x_{M}$. For this reason, we shall often write $I_1^{(M)}(x_{1}, \ldots,x_{M})$.

In words, $I_{1}^{(M)}(\mathbf{x})$ tests whether the first entries of the DFS sequence $\mathbf{x}$ encode a tree of size at most $M$ whose $G$-value is $1$, and $I_{i}^{(M)}(\mathbf{x})$ is defined in the same way but after $i$ cyclic shifts. The parameter $M$ serves as a cutoff threshold on the size of the tree. Also observe that in the particular case $G(T)=\mathbbm{1}_{T=t}$ (for a fixed tree $t$) this gives fringe tree counts, but the function $G$ can be quite general.

If $T$ is a tree, observe that
$$\sum_{i=1}^{n} I^{(M)}_{i}(\mathbf{x}(T))$$
counts the number of fringe subtrees of $T$ of size at most $M$ whose $G$-value is $1$.  When $M<n$, we call this a truncated count. The following result immediately follows from \eqref{eq:shifts}.

\begin{lemma}
\label{lem:RW}
We have
$$\sum_{i=1}^{n} I^{(M)}_{i}(\mathbf{x}(\T^{n})) \quad \mathop{=}^{(d)} \qquad \sum_{i=1}^{n} I^{(M)}_{i}(X_{1}, \ldots,X_{n}) \quad \textrm{ under } \quad  \P(\cdot \mid S_{n}=-1).$$
\end{lemma}

\begin{figure}[h]
    \centering
    \includegraphics[width=0.6\linewidth]{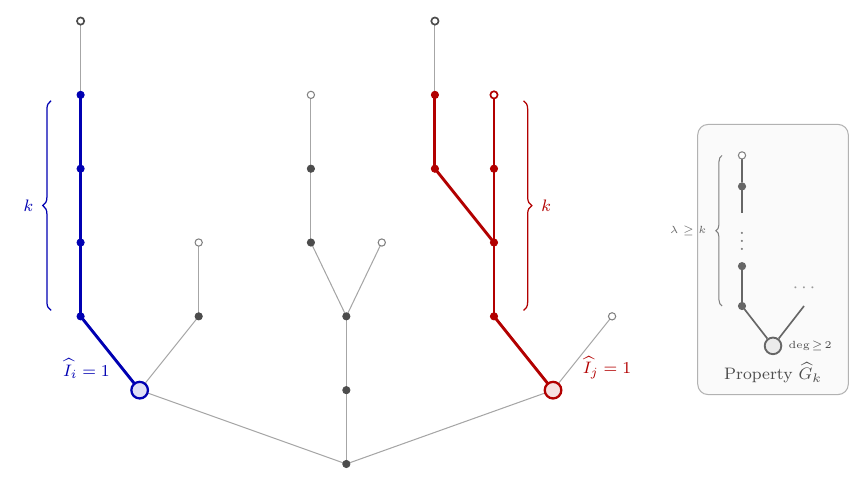}
\caption{Illustration of a declumped leaf-height indicator. The marked vertices satisfy the condition $k_v(T)\ge2$ and have a child whose fringe subtree has root leaf-height at least $k$. Such indicators select boundary witnesses inside clusters of raw leaf-height occurrences.} 
\label{fig:indicator_illustration}
\end{figure}

The following consequence will be useful.

\begin{corollary}
\label{cor:cyclicshift}
Let $G$ be a $ \{0,1\}$-valued function defined on the set of all finite plane trees.  Let $U_n$ be a uniform vertex of $\mathcal T^n$, independent of
$\mathcal T^n$. Set also
 $\Psi_m(i)=\P(S_m=i)$, with the convention $\Psi_m(i)=0$ if $m<0$. Then
$$
\P(G(\mathcal T^n_{U_n})=1)
=
\frac{1}{\Psi_n(-1)}
\E\left[
\mathbbm{1}_{\{G(\mathcal T)=1\}}
\Psi_{n-|\mathcal T|}(0)
\right].
$$
\end{corollary}

\begin{proof}
By the cyclic-shift representation \eqref{eq:shifts}, the DFS sequence of
$\mathcal T^n_{U_n}$ has the same law as the initial tree encoded by
$(X_1,\ldots,X_n)$ under $\P(\cdot\mid S_n=-1)$. Hence
$$
\P(G(\mathcal T^n_{U_n})=1)
=
\P(I_1^{(n)}(X_1,\ldots,X_n)=1\mid S_n=-1),
$$
where the indicator is defined with the function $G$. Therefore
$$
\P(G(\mathcal T^n_{U_n})=1)
=
\frac{
\P(I_1^{(n)}(X_1,\ldots,X_n)=1,\ S_n=-1)
}{\Psi_n(-1)}.
$$
On the event $\{\zeta=m\}$, the first $m$ increments encode a Bienaym\'e tree
$\mathcal T$ with $|\mathcal T|=m$, and after time $m$ the walk starts from
$-1$. By the Markov property at time $\zeta$,
$$
\P(I_1^{(n)}=1,\ S_n=-1)
=
\sum_{m\ge1}
\P(G(\mathcal T)=1,\ |\mathcal T|=m)\Psi_{n-m}(0).
$$
With the convention $\Psi_{n-m}(0)=0$ when $m>n$, the last sum is exactly$
\E\left[
\mathbbm{1}_{\{G(\mathcal T)=1\}}
\Psi_{n-|\mathcal T|}(0)
\right]$.
Dividing by $\Psi_n(-1)$ gives the result.
\end{proof}

For the proof of Theorems \ref{thm:max_lh} and \ref{thm:max_string}, we shall apply the preceding cyclic-shift representation to functions $G_k$ which encode the existence of large non-fringe subtrees inside the corresponding fringe subtree. More precisely, raw indicators naturally describe the events $\lambda(T)\ge k$ and $H_r(T)\ge k$: if $G_k(T)=\mathbbm{1}_{\{\lambda_\varnothing(T)\ge k\}}$, then $\lambda(T)\ge k$ if and only if $\sum_{v\in T}G_k(T_v)\ge1$; similarly, if $G_k(T)=\mathbbm{1}_{\{H_{r,k}\prec T\}}$, then $H_r(T)\ge k$ if and only if $\sum_{v\in T}G_k(T_v)\ge1$.

However, these raw non-fringe occurrences may form clusters, and a direct Poisson approximation for their counts is in general false. We therefore introduce declumped indicators, which count suitable boundary points of such clusters (see Figure \ref{fig:indicator_illustration} for an illustration of  declumping). For the leaf-height problem we shall use indicators of the form
$\mathbbm{1}_{\{k_\varnothing(T)\ge\kappa+1,\ \exists u\in\Gamma_\varnothing(T):\lambda_\varnothing(T_u)\ge k\}}$, where $\kappa=\min\{j\ge1:\mu_j>0\}$, and for the complete $r$-ary problem we shall use indicators of the form
$\mathbbm{1}_{\{k_\varnothing(T)\neq r,\ \exists u\in\Gamma_\varnothing(T):H_{r,k}\prec T_u\}}$.

These declumped indicators do not, in general, give a deterministic equivalence with the maximum events. Rather, they provide Poisson lower-bound witnesses, while the corresponding raw counts provide first-moment upper bounds. Combining these two estimates yields the asymptotics of $\lambda(\mathcal T^n)$ and $H_r(\mathcal T^n)$.

\subsection{Random walk toolbox}
\label{sec:toolbox}

All the rest of the paper relies on the use of conditioned random walks. In this short part, we gather notation and results that will be extensively used in the sequel.

We always assume that $\mu$ is critical and belongs to the domain of attraction of an $\alpha$-stable law with $\alpha \in (1,2]$ and that $\mathbf{X}=(X_i)_{i \geq 1}$ is a sequence of i.i.d. random variables with law given by $\P(X_{1}=i)=\mu_{i+1}$ for $i \geq -1$. As before,
$$S_0 = 0, \qquad S_k = \sum_{i=1}^k X_i, \qquad k \geq 1.$$
We additionally assume that $\mu$ is aperiodic for simplicity (all the results can be extended to the periodic case by working along suitable subsequences).

There exists an increasing sequence $(a_{n})$ such that $S_{n}/a_{n}$ converges in distribution to a random variable $Y_{\alpha}$ with Laplace exponent given by $\Es{ \exp(- \lambda Y_{\alpha})}= \exp(\lambda ^ \alpha)$ for every $\lambda>0$ (see e.g.~\cite[Sec.~XVII.5]{Fel71}). In addition, the convergence
\begin{equation}
\label{eq:cvlaw}  \left(\frac{S_{\lfloor  n t \rfloor}}{a_{n}} : 0 \leq t \leq  1 \right) \quad \mathop{\longrightarrow}^{(d)}_{n \rightarrow \infty} \quad (Y_{\alpha}(t): 0 \leq t \leq 1)
\end{equation}
holds in distribution for the Skorokhod $J_{1}$ topology on $[0,1]$, where $ (Y_{\alpha}(t): 0 \leq t \leq 1)$ is the spectrally positive $\alpha$-stable Lévy process with $Y_{\alpha}(1)$ having the same law as $Y_{\alpha}$. In addition, $(a_{n})$ is regularly varying of index $1/\alpha$, meaning that $a_{n}/n^{1/\alpha}$ is slowly varying. By definition, a slowly varying function $\ell : \R_{+} \rightarrow\R$ satisfies $\ell(ax)/\ell(x) \rightarrow1$ as $x \rightarrow\infty$ for every $a>0$ (see \cite{bingham1989regular} for background on slowly varying functions and sequences). We will often use the following result for slowly varying functions (often called Potter bounds, see \cite[Theorem 1.5.6]{bingham1989regular}). Let $\ell(x)$ be a slowly varying function at infinity. For every $A>1$, $\delta>0$, there exists $B>0$ such that for every $x,y\geq B$:
\begin{equation}\label{eq: potter_bound}
    \frac{\ell(y)}{\ell(x)}\leq A \max \big( (y/x)^{\delta}, (y/x)^{-\delta}\big)
\end{equation}

\paragraph{Periodicity.}
Let
\[
d=\gcd\{i-j:\mu_i\mu_j>0\}
\]
be the span of the support of $\mu$.  When $d=1$, $\mu$ is called aperiodic, and in this case $\mathbb P(|\mathcal T|=n)>0$ for every $n$ sufficiently large and the local limit theorem  \cite[Theorem 4.2.1]{IL71} states that
\begin{equation}
\label{eq:LL}\sup_{k \in \Z} \left| a_{n} \Pr{S_{n}=k}- d_{\alpha} \left(  \frac{k}{a_{n}} \right)  \right|  \quad \mathop{\longrightarrow}_{n \rightarrow \infty} \quad 0
\end{equation} 
where $d_{\alpha}$ is the density of $Y_{\alpha}$. When $\mu$ has finite variance, $a_{n} \sim c \sqrt{n}$ for some constant $c>0$.

For notational simplicity, the proofs below are written in the aperiodic case. In the periodic case, the only modification in the proofs would be to replace the aperiodic local limit theorem
\eqref{eq:LL}  by its lattice version:
\[
\sup_{k\in n b+d\mathbb Z}
\left|
a_n\mathbb P(S_n=k)-d\,d_\alpha(k/a_n)
\right|
 \quad \mathop{\longrightarrow}_{n \rightarrow \infty} \quad 0,
\]
where $b$ is any element of the support of $X_1$ modulo $d$; outside the
lattice $n b+d\mathbb Z$ the probability is zero. Since all bridge
probabilities appearing below are evaluated only at compatible lattice points,
the extra factor $d$ cancels in all ratios such as ${\Psi_{n-m}(0)}/{\Psi_n(-1)}$.

Consequently the estimates and conclusions proved below in the aperiodic case
remain valid in the periodic case along admissible values of $n$.

\paragraph{Useful notation.} For every $n \geq 1$,  let $\zeta=\zeta((X_{i})_{i \geq 1})$ be  the first hitting time of $-1$ by $(S_{k})_{k \geq 1}$. We consider a sequence $(G_k)$  of $\{0,1\}$-valued functions on finite plane trees and let $(k_{n})$ be a sequence of integers. We set 
$$\pi_n=\P(G_{k_n}(\mathcal T)=1).$$

\section{Direct applications and consequences of Assumption \ref{hyp:S}}
\label{sec:applications}
 
In this section we use Theorem \ref{thm:PoissonApproximation} as a black box
and record several direct applications. We first treat genuine fringe subtree
counts. In this case the marking fixes the size of the marked subtree, so the
local clustering assumptions are automatic and Assumption \ref{hyp:M} can be
checked directly.

We then turn to applications where Assumption \ref{hyp:M} is verified through
the structural Assumption \ref{hyp:S}. We first prove that \ref{hyp:S}
implies the required microscopicity estimates, and record a few consequences
that will be used later. This allows us to handle raw non-fringe subtree
counts in a low-clumping regime, as well as counts of vertices with unusually
large prescribed outdegree.

As previously mentioned, we implicitly assume in the proofs that $\mu$ is aperiodic.

\subsection{Applications of Theorem \ref{thm:PoissonApproximation}  to fringe subtree counts}
\label{ssec:fringe}

We record here two direct consequences of Theorem
\ref{thm:PoissonApproximation}. The first one concerns genuine fringe subtree
counts. In this setting, if the marking function fixes the size of the marked
subtree, then local clustering is automatically excluded: a fringe subtree of
size $k_n$ cannot contain a proper fringe subtree of the same size. Thus
Assumption \ref{hyp:C} becomes automatic, and Theorem
\ref{thm:PoissonApproximation} gives a clean extension of several classical
Poisson approximation results for fringe subtree counts to offspring
distributions in the stable domain of attraction. Recall that $\pi_n=\P(G_{k_n}(\mathcal T)=1)$.

\begin{corollary}
\label{cor:Poisson_fringe}
Assume that $\mu$ is critical and belongs to the domain of attraction
of an $\alpha$-stable law for some $\alpha\in(1,2]$. Let $(k_n)_{n\geq 0}$ be
an integer sequence such that $k_n\to\infty$ and $k_n=o(n)$.
Suppose that, for every $k\ge1$ and every plane tree $\tau$, $
G_k(\tau)=1
\implies
|\tau|=k$.
Then:
\begin{enumerate}
\item[(i)] If $(n\pi_n)$ is bounded, then
\begin{equation}
\label{eq:DTVfringe}
\dTV \left(
\# \{u \in \mathcal{T}^{n} : G_{k_{n}}( \mathcal{T}^{n}_{u})=1\},
\Po(n \pi_{n})
\right)
\quad \mathop{\longrightarrow}_{n \rightarrow \infty} \quad 0.
\end{equation}

\item[(ii)] If $n\pi_n\to\infty$, then
$$
\frac{
\#\{u\in\mathcal T^n: G_{k_n}(\mathcal T^n_u)=1\}
}{n\pi_n}
\quad \mathop{\longrightarrow}^{(\P)}_{n \rightarrow \infty} \quad 1.
$$
In particular,
$$
\#\{u\in\mathcal T^n: G_{k_n}(\mathcal T^n_u)=1\}
\quad \mathop{\longrightarrow}^{(\P)}_{n \rightarrow \infty} \quad \infty.
$$
\end{enumerate}
\end{corollary}

\begin{proof}
Take the cutoff $M_n=k_n$. Since $k_n\to\infty$ and $k_n=o(n)$, it is admissible. If $U_n$ is a uniform vertex of $\mathcal T^n$, Corollary \ref{cor:cyclicshift} and the implication $G_{k_n}(\tau)=1\Rightarrow |\tau|=k_n$ give
\[
\P(G_{k_n}(\mathcal T^n_{U_n})=1)
=
\P(G_{k_n}(\mathcal T)=1)\frac{\Psi_{n-k_n}(0)}{\Psi_n(-1)}
\sim \pi_n,
\]
and the conditioned marked fringe subtree has size exactly $M_n$. Hence Assumption \ref{hyp:M} holds. Moreover, a fringe subtree of size $k_n$ cannot contain a proper descendant fringe subtree of the same size, so $A_n(\ell)=0$ for every $\ell\ge1$; Assumption \ref{hyp:C} follows. The two conclusions are therefore exactly Theorem \ref{thm:PoissonApproximation}\textup{(i)} and \textup{(ii)}.
\end{proof}

\begin{remark}\label{rem:implication_fringe}
Assume here that $k_n=o(n)$ and $k_n\to\infty$.
Corollary \ref{cor:Poisson_fringe} recovers Theorem 1.3 \textup{(i)} and
\textup{(ii)} of Cai and Devroye~\cite{CD17} by taking
$G_{k_n}(T):=\mathbbm{1}_{\{T=T_n\}}$ for a sequence of plane trees
$(T_n)$ with $|T_n|=k_n$. It also recovers Theorem 1.4 \textup{(i)} and
\textup{(ii)} of \cite{CD17} by taking
$G_{k_n}(T):=\mathbbm{1}_{\{|T|=k_n\}}$.

Furthermore, let $\mathbf n_k=(n_k(i))_{i\ge0}$ be a sequence of
non-negative integers. We say that $\mathbf n_k$ is a degree statistic of
size $k$ if
$$
\sum_{i\ge0} n_k(i)
=
1+\sum_{i\ge0} i n_k(i)
=
k.
$$
Denote by $\mathbbm T_{\mathbf n}$ the set of plane trees with degree
statistic $\mathbf n$, that is, $T\in\mathbbm T_{\mathbf n}$ if
$$
n(i)=\#\{v\in T:k_v(T)=i\}
\qquad\textrm{for all } i\ge0.
$$
Let $\mathbf n_{k_n}$ be a degree statistic of size $k_n$. Taking
$$
  G_{k_n}(T)
  =\mathbbm{1}_{\{T\in\mathbbm T_{\mathbf n_{k_n}}\}},
$$
Corollary \ref{cor:Poisson_fringe} yields the conditioned-Bienaym\'e
analogue of the Poisson part of
Berzunza--Holmgren--Janson~\cite[Theorem~4.2]{BHJ26}. Their theorem is
stated for a uniform tree with a deterministic degree statistic; the
transfer to conditioned Bienaym\'e trees is discussed in their Section~6.
\end{remark}

 \subsection{A criterion for \ref{hyp:M} and for \ref{hyp:S}}
 
 The following result gives a simple criterion that implies \ref{hyp:M}. Although it is not used below, we record it for independent interest.
   
\begin{lemma}
\label{lem:M}
Let $(M_n)$ be a sequence of positive integers such that
$M_n\to\infty$ and $M_n=o(n)$. Assume that $\pi_n>0$ for all sufficiently
large $n$ and that
$$
a_n\,\P\bigl(|\mathcal T|>M_n\mid G_{k_n}(\mathcal T)=1\bigr)
\longrightarrow0.
$$
Then $(G_k)$ satisfies Assumption \ref{hyp:M} along $(k_n)$ with respect to
$(M_n)$.
\end{lemma}

\begin{proof}
Write $\pi_n=\P(G_{k_n}(\mathcal T)=1)$ and set
$$
R_n(m)=\frac{\Psi_{n-m}(0)}{\Psi_n(-1)}.
$$
The local limit theorem gives $R_n(m)=1+o(1)$ uniformly for
$m\le M_n=o(n)$ as well as $R_n(m)\le Ca_n$ for every $m$. Hence
$$
\E\left[R_n(|\mathcal T|)
\mathbbm{1}_{\{|\mathcal T|>M_n\}} \mid G_{k_n}(\mathcal T)=1 \right]
\le Ca_n\P(|\mathcal T|>M_n \mid G_{k_n}(\mathcal T)=1)=o(1),
$$
and the assumption also gives $\P(|\mathcal T|>M_n \mid G_{k_n}(\mathcal T)=1)=o(1)$.
Corollary \ref{cor:cyclicshift} therefore yields
$$
\P(G_{k_n}(\mathcal T^n_{U_n})=1)
=\pi_n\E[R_n(|\mathcal T|)  \mid G_{k_n}(\mathcal T)=1]\sim\pi_n,
$$
while the same display restricted to $\{|\mathcal T|>M_n\}$ is
$o(\pi_n)$. Dividing the latter estimate by the former proves the
conditional cutoff in Assumption \ref{hyp:M}.
\end{proof}
 
The following result gives a simple sufficient condition for Assumption
\ref{hyp:S}  when the cutoff is of the
form $M_n=\lfloor n^{1-\eta}\rfloor$.

\begin{lemma}
\label{lem:critM}
Set $
  \P_k(\,\cdot\,)
  \coloneqq
  \P(\,\cdot\mid G_k(\mathcal T)=1)$ and $  \E_k[\,\cdot\,]
  \coloneqq
  \E[\,\cdot\mid G_k(\mathcal T)=1]$. Assume that
there exist nonnegative integer-valued random variables $C_k,L_k$ such that
one can construct copies $\mathcal T^{(1)},\mathcal T^{(2)},\ldots$ which,
conditionally on $(C_k,L_k)$, are independent and each have the same law
as $\mathcal T$, with
$$
  |\mathcal T|
  =
  C_k+\sum_{i=1}^{L_k}|\mathcal T^{(i)}|
  \qquad\text{under }\P_k \qquad  \textrm{and} \qquad 
  \E_{k_n}[C_{k_n}+L_{k_n}]
  =
  n^{o(1)}.
$$
Then, for every $\eta\in(0,1-1/\alpha)$, Assumption \ref{hyp:S} holds for
$(G_k)$ along $(k_n)$ with respect to $
  M_n=\lfloor n^{1-\eta}\rfloor$.
\end{lemma}

\begin{proof}
Fix $\eta\in(0,1-1/\alpha)$ and set
$M_n=\lfloor n^{1-\eta}\rfloor$. The required decomposition is already part
of the assumptions, so it remains only to verify the two quantitative
conditions in Assumption \ref{hyp:S}.

Since $(a_n)$ is regularly varying with index $1/\alpha$, we have $a_n=n^{1/\alpha+o(1)}$ and
$  a_{M_n}
  =
  n^{(1-\eta)/\alpha+o(1)}$.
For every fixed $\varepsilon>0$, Markov's inequality gives
$$  a_n\P_{k_n}(C_{k_n}>\varepsilon M_n) \le
  \frac{a_n}{\varepsilon M_n}\E_{k_n}[C_{k_n}]=
  n^{1/\alpha-1+\eta+o(1)}.$$
  Since $\eta<1-1/\alpha$, the exponent
$1/\alpha-1+\eta$ is negative, and therefore $
  a_n\P_{k_n}(C_{k_n}>\varepsilon M_n)
\rightarrow 0$.
Similarly,
$$
  \frac{\E_{k_n}[L_{k_n}]}{a_{M_n}}
  \le
  \frac{n^{o(1)}}{n^{(1-\eta)/\alpha+o(1)}}
  =
  n^{-(1-\eta)/\alpha+o(1)}
  \longrightarrow0.
$$
Thus Assumption \ref{hyp:S} holds along $(k_n)$ with respect to $(M_n)$.
\end{proof}
 
\subsection{Some consequences of Assumption \ref{hyp:S}}
\label{ssec:consequenceS}

Here we explore some consequences of Assumption \ref{hyp:S}. We first need a
technical estimate for Bienaym\'e forests and then derive a useful
microscopicity consequence of \ref{hyp:S}. As before, we write
$\Psi_m(i)=\P(S_m=i)$, with the convention $\Psi_m(i)=0$ if $m<0$.

  \begin{lemma}
\label{lem:forest_estimates}
Let $\mathcal G_\ell$ be a forest of $\ell$ independent Bienaym\'e trees.
There exists a constant $C>0$ such that, uniformly in integers
$\ell\ge1$ and $1\le m\le n/2$,
$$
\P(|\mathcal G_\ell|>m)\le C\frac{\ell}{a_m}
\quad\text{and}\quad
\frac{1}{\Psi_n(-1)}
\E\left[
\sup_{x\in\mathbb Z}\Psi_{n-|\mathcal G_\ell|}(x)
\mathbbm{1}_{\{|\mathcal G_\ell|>m\}}
\right]
\le C\frac{\ell}{a_m}.
$$
\end{lemma}

\begin{proof}
By Dwass' formula, for every $q\ge1$,
$\P(|\mathcal G_\ell|=q)=(\ell/q)\P(S_q=-\ell)$, with the convention that
this quantity is $0$ when the event is impossible. The local bound following
from \eqref{eq:LL} gives $\P(S_q=-\ell)\le C/a_q$, uniformly in $q$ and
$\ell$. Hence
$$
\P(|\mathcal G_\ell|>m)
\le
C\ell\sum_{q>m}\frac{1}{q a_q}
\le
C\frac{\ell}{a_m},
$$
by Karamata's theorem.

For the second estimate, write
$$
\E\left[
\sup_x\Psi_{n-|\mathcal G_\ell|}(x)
\mathbbm{1}_{\{|\mathcal G_\ell|>m\}}
\right]
=
\sum_{q>m}\P(|\mathcal G_\ell|=q)\sup_x\Psi_{n-q}(x),
$$
where $\Psi_{n-q}\equiv0$ if $q>n$. We split the sum into $m<q\le n/2$ and
$n/2<q\le n$. For $q\le n/2$, the local bound and regular variation give
$\sup_x\Psi_{n-q}(x)\le C/a_n$. Therefore this part is bounded by
$C a_n^{-1}\P(|\mathcal G_\ell|>m)$, and after division by
$\Psi_n(-1)\asymp a_n^{-1}$ it is at most $C\ell/a_m$.

For $q>n/2$, put $h=n-q$. Using Dwass' formula and the local bound again,
this contribution divided by $\Psi_n(-1)$ is at most
$$
C a_n\ell
\sum_{q>n/2}\frac{1}{q a_q}\sup_x\Psi_{n-q}(x).
$$
Since $q\asymp n$ and $a_q\asymp a_n$ on this range, this is bounded by
$$
\frac{C\ell}{n}\sum_{h=0}^{n/2}\sup_x\Psi_h(x).
$$
As $\sup_x\Psi_h(x)\le C/a_h$ for $h\ge1$ and $\Psi_0(0)=1$, Karamata's
theorem gives $\sum_{h=0}^{n/2}\sup_x\Psi_h(x)\le Cn/a_n$. Hence the
contribution of the range $q>n/2$ is at most $C\ell/a_n\le C\ell/a_m$, since
$m\le n/2$. This completes the proof.
\end{proof}

\begin{lemma}
\label{lem:consequenceM}
Let $(G_k)$ be a sequence of $\{0,1\}$-valued functions on finite plane trees,
and fix a sequence $(k_n)$. Assume that $(G_k)$ satisfies Assumption
\ref{hyp:S} along $(k_n)$ with respect to a cutoff sequence $(M_{n})$.   Then
\begin{equation}
\label{eq:record} \P(|\mathcal T|>M_n/4 \mid G_{k_n}(\mathcal T)=1)\to0 \quad \textrm{and} \quad 
 \E\left[
\frac{\sup_x\Psi_{n-|\mathcal T|}(x)}{\Psi_n(-1)}
\mathbbm{1}_{\{|\mathcal T|>M_n/4\}} \mid G_{k_n}(\mathcal T)=1
\right]\to0.
\end{equation}
\end{lemma}

\begin{proof}
Write $
  \P_{k_n}(\,\cdot\,)
  =
  \P(\,\cdot\mid G_{k_n}(\mathcal T)=1)$ and   $\E_{k_n}[\,\cdot\,]
  =
  \E[\,\cdot\mid G_{k_n}(\mathcal T)=1]$.
By Assumption \ref{hyp:S}, under $\P_{k_n}$ we may write
$|\mathcal T|=C_n+|\mathcal F_{L_n}|$, where
$C_n=C_{k_n}$, $L_n=L_{k_n}$, and, conditionally on $(C_n,L_n)$,
$\mathcal F_{L_n}$ is a forest of $L_n$ independent Bienaym\'e trees.

We first prove the unweighted estimate. Since
$\{|\mathcal T|>M_n/4\}$ is contained in
$\{C_n>M_n/8\}\cup\{|\mathcal F_{L_n}|>M_n/8\}$, we have
$
  \P_{k_n}(|\mathcal T|>M_n/4)
  \le
  \P_{k_n}(C_n>M_n/8)
  +
  \P_{k_n}(|\mathcal F_{L_n}|>M_n/8)$.
The first term tends to $0$, since Assumption \ref{hyp:S} gives
$a_n\P_{k_n}(C_n>M_n/8)\to0$ and $a_n\to\infty$. For the second term,
Lemma \ref{lem:forest_estimates}, conditionally on $L_n$, and regular
variation of $(a_m)$ give
\[
  \P_{k_n}(|\mathcal F_{L_n}|>M_n/8)
  \le
  C\frac{\E_{k_n}[L_n]}{a_{M_n}}
  \longrightarrow0.
\]
This proves the first assertion of \eqref{eq:record}.

For the bridge-weighted estimate, split again according to whether
$C_n>M_n/8$. Since $\sup_x\Psi_{n-|\mathcal T|}(x)\le1$ and
$\Psi_n(-1)\sim c/a_{n}$ for some $c>0$, the contribution of this event is at most
$Ca_n\P_{k_n}(C_n>M_n/8)=o(1)$.

On $\{C_n\le M_n/8\}$, the event $|\mathcal T|>M_n/4$ implies
$|\mathcal F_{L_n}|>M_n/8$. Uniformly for $c\le M_n/8$, 
$\Psi_{n-c}(-1)\sim \Psi_n(-1)$. Moreover, $M_n/8\le (n-c)/2$ for all large
$n$. Therefore Lemma \ref{lem:forest_estimates}, applied with $n-c$ in place
of $n$, yields
\[
  \E\left[
    \frac{\sup_x\Psi_{n-c-|\mathcal F_\ell|}(x)}{\Psi_n(-1)}
    \mathbbm{1}_{\{|\mathcal F_\ell|>M_n/8\}}
  \right]
  \le
  C\frac{\ell}{a_{M_n}}.
\]
Taking expectations with respect to $(C_n,L_n)$ shows that the contribution
of $\{C_n\le M_n/8\}$ is bounded by
$C\E_{k_n}[L_n]/a_{M_n}=o(1)$. Combining the two contributions proves the
second assertion of \eqref{eq:record}.
\end{proof}

 We  mention a  consequence of Lemma \ref{lem:consequenceM} which will be useful later.
 
\begin{corollary}
\label{cor:SM}
Let $(G_k)$ be a sequence of $\{0,1\}$-valued functions on finite plane
trees, let $(k_n)$ be a sequence of integers, and let $(M_n)$ be a cutoff
sequence. Assume that $(G_k)$ satisfies Assumption \ref{hyp:S} along $(k_n)$
with respect to $(M_n)$. Then:
\begin{enumerate}
\item[(i)] $(G_k)$ satisfies Assumption \ref{hyp:M} along $(k_n)$ with
respect to $(M_n)$;
\item[(ii)] we have
$
\E\left[
\#\{u\in\mathcal T^n:G_{k_n}(\mathcal T^n_u)=1\}
\right]
\sim
n\P(G_{k_n}(\mathcal T)=1)
$;
\item[(iii)] if $\P(G_{k_n}(\mathcal T)=1)=O(1/n)$, then
$\P(G_{k_n}(\mathcal T^n)=1)\to0$.
\end{enumerate}
\end{corollary}

\begin{proof}
Set $\pi_n=\P(G_{k_n}(\mathcal T)=1)$. We may assume that $\pi_n>0$ for all
large $n$, and write
$\P_n(\,\cdot\,)=\P(\,\cdot\mid G_{k_n}(\mathcal T)=1)$ and $\E_n$ for the
corresponding expectation.

Let $U_n$ be a uniform vertex of $\mathcal T^n$, independent of
$\mathcal T^n$. By Corollary \ref{cor:cyclicshift},
$$
\P(G_{k_n}(\mathcal T^n_{U_n})=1)
=
\pi_n\E_n\left[
\frac{\Psi_{n-|\mathcal T|}(0)}{\Psi_n(-1)}
\right].
$$
On $\{|\mathcal T|\le M_n\}$, the ratio in the expectation converges
uniformly to $1$, since $M_n=o(n)$. The first estimate of
Lemma \ref{lem:consequenceM} shows that
$\P_n(|\mathcal T|>M_n)\to0$, while its second estimate shows that the
bridge-weighted contribution of $\{|\mathcal T|>M_n\}$ tends to $0$.
Consequently,
\begin{equation}
\label{eq:Fthetan}
\P(G_{k_n}(\mathcal T^n_{U_n})=1)\sim\pi_n.
\end{equation}
Moreover, again by Lemma \ref{lem:consequenceM},
$
\P(
G_{k_n}(\mathcal T^n_{U_n})=1,\
|\mathcal T^n_{U_n}|>M_n
)
=o(\pi_n)$.
Dividing by \eqref{eq:Fthetan} proves Assumption \ref{hyp:M}, and hence
\textup{(i)}. Part \textup{(ii)} follows from
$$
\E\left[
\#\{u\in\mathcal T^n:G_{k_n}(\mathcal T^n_u)=1\}
\right]
=
n\P(G_{k_n}(\mathcal T^n_{U_n})=1)
$$
and \eqref{eq:Fthetan}.

It remains to prove \textup{(iii)}. Assume that $\pi_n=O(1/n)$. By
conditioning on $|\mathcal T|=n$,
$$
\P(G_{k_n}(\mathcal T^n)=1)
=
\pi_n\frac{\P_n(|\mathcal T|=n)}{\P(|\mathcal T|=n)}.
$$
By Dwass' formula and the local limit theorem,
$\P(|\mathcal T|=n)\asymp 1/(na_n)$ along admissible values of $n$.

Under $\P_n$, write
$|\mathcal T|=C_n+\sum_{i=1}^{L_n}|\mathcal T^{(i)}|$, where
$C_n=C_{k_n}$ and $L_n=L_{k_n}$. On $\{C_n\le n/2\}$, conditionally on
$(C_n,L_n)=(c,\ell)$, Dwass' formula and the local limit bound give, for
$\ell\ge1$,
$$
\P\left(\sum_{i=1}^{\ell}|\mathcal T^{(i)}|=n-c\right)
=
\frac{\ell}{n-c}\P(S_{n-c}=-\ell)
\le
\frac{C\ell}{na_n}.
$$
It follows that the contribution of $\{C_n\le n/2\}$ to
$\P(G_{k_n}(\mathcal T^n)=1)$ is bounded by
$C\pi_n\E_n[L_n]$. By Assumption \ref{hyp:S},
$\E_n[L_n]=o(a_{M_n})$. Since $(a_m)$ is regularly varying with index
$1/\alpha<1$, we have $a_{M_n}=o(M_n)=o(n)$, and therefore
$$
\pi_n\E_n[L_n]
=
o\left(\frac{a_{M_n}}{n}\right)
=
o(1).
$$

On the other hand, since $M_n=o(n)$, for all large $n$ the event
$\{C_n>n/2\}$ is contained in $\{C_n>M_n\}$. Its contribution is bounded by
$
C\pi_nna_n\P_n(C_n>n/2)
\le
C a_n\P_n(C_n>M_n)
$,
which tends to $0$ by Assumption \ref{hyp:S}, applied with
$\varepsilon=1$. This proves \textup{(iii)}.
\end{proof}

\subsection{Applications of Theorem \ref{thm:PoissonApproximation} using Assumption \ref{hyp:S}}

Here, taking Theorem \ref{thm:PoissonApproximation}  for granted, we give two applications obtained by checking assumption \ref{hyp:S} (which implies \ref{hyp:M} by Corollary \ref{cor:SM} (i)).

The first one concerns raw non-fringe subtree counts. Recall that, for two plane
trees $T'$ and $T$, the notation $T'\prec T$ means that $T$ can be obtained
from $T'$ by grafting plane trees onto the leaves of $T'$. Applying Theorem
\ref{thm:PoissonApproximation} with
$$
G_{k_n}(T):=\mathbbm{1}_{\{T_{k_n}\prec T\}}
$$
gives a Poisson approximation for
$$
\#\{u\in\mathcal T^n : T_{k_n}\prec \mathcal T^n_u\},
$$
provided that overlapping occurrences do not create significant local
clusters. Unlike in the fringe case, this low-clumping condition is not
automatic.

\begin{corollary}\label{cor:Non-fringe Poisson}
    Assume that $\mu$ is critical, aperiodic, and belongs to the domain of
    attraction of an $\alpha$-stable law for some $\alpha\in(1,2]$. Let $(T_k)_{k\geq 1}$ be a sequence of plane trees with  $|T_{k}|=k$.\ Let $L_k\coloneqq \#\{v\in T_k:k_v(T_k)=0\}$ be the number of leaves, and $C_k\coloneqq k-L_k$ be the number of internal vertices in $T_k$. Let $(k_n)$ be an integer sequence with $k_n \to \infty$, and set $\pi_n = \P(T_{k_n}\prec \T)$. Let $(M_n)$ be a cutoff sequence and assume that
    $$
      C_{k_n}=o(M_n),
      \quad
      L_{k_n}=o\bigl(a_{M_n}\bigr), \quad \sup_{\ell\ge1} \P\bigl(
    |\mathcal T|\ge \ell+1,\
    T_{k_n} \prec\mathcal T,\
    T_{k_n} \prec \T_{u_{\ell+1}(\mathcal T)}
    \bigr) = o\Big( \frac{n}{M_n}\pi_n^2\Big ).
    $$
Then:
\begin{enumerate}
\item[(i)]
    If $(n\pi_n)$ is bounded 
    then
    \[
    \dTV \left(
            \#\{u\in\mathcal T^n : T_{k_n}\prec \mathcal T^n_u\},
            \Po(n\pi_n)
            \right)
            \quad \mathop{\longrightarrow}_{n \rightarrow \infty} \quad 0.
    \]
    \item[(ii)] If $n\pi_n\to\infty$ 
    then
$$
\frac{
\#\{u\in\mathcal T^n:T_{k_n}\prec\mathcal T^n_u\}
}{
n\pi_n
}
\quad \mathop{\longrightarrow}^{(\P)}_{n \rightarrow \infty} \quad 1.
$$
    \end{enumerate}
\end{corollary}

Corollary \ref{cor:Non-fringe Poisson} should be viewed as a low-clumping
version of Poisson approximation for raw non-fringe counts. It extends the
scope of Theorem 1.6 of Cai and Devroye~\cite{CD17} to offspring distributions
in the stable domain of attraction, but only in regimes where overlapping
non-fringe occurrences are sufficiently rare. The complete unary case shows
that such an assumption is necessary in general: a single long unary chain
creates a whole cluster of overlapping non-fringe occurrences.

\begin{proof}[Proof of Corollary \ref{cor:Non-fringe Poisson}]
Set $G_k(T)=\mathbbm{1}_{\{T_k\prec T\}}$ and $\P_k(\,\cdot\,)\coloneqq\P(\,\cdot\mid G_k(\mathcal T)=1)$. We apply Theorem \ref{thm:PoissonApproximation}.

First, $(G_k)$ satisfies Assumption \ref{hyp:S}: conditionally on $T_{k_n}\prec\mathcal T$, the tree is obtained from $T_{k_n}$ by grafting independent Bienaym\'e trees onto the leaves of $T_{k_n}$. Thus, we have $|\T| = C_{k} + \sum_{i=1}^{L_k}|\mathcal T^{(i)}|$ under $\P_k$, where $\mathcal T^{(1)},\mathcal T^{(2)},\dots$ are i.i.d. copies of $\T$. Since $C_{k_n}$ and $L_{k_n}$ are deterministic,
$C_{k_n}=o(M_n)$ implies that, for every $\varepsilon>0$,
$a_n\P_{k_n}(C_{k_n}>\varepsilon M_n)=0$ 
for all sufficiently large $n$, while
$L_{k_n}/a_{M_n}\to0$. Hence $(G_k)$ satisfies Assumption
\ref{hyp:S} along $(k_n)$ with respect to $(M_n)$.

The third equality in the assumptions made on $(M_{n})$ is precisely Assumption \ref{hyp:C}, and the conclusion follows from  Theorem
\ref{thm:PoissonApproximation}.
\end{proof}

Another application of Theorem \ref{thm:PoissonApproximation} using Assumption \ref{hyp:S} concerns
Poisson approximations for vertices of unusually large prescribed outdegree. The short argument below illustrates the flexibility of our framework.

\begin{proposition}
\label{prop:largeoutdegree}
Assume that $\mu$ is critical and belongs to the domain of
    attraction of an $\alpha$-stable law for some $\alpha\in(1,2]$. Let $(d_n)$ be a sequence of positive integers such that
$d_n\to\infty$, and $\mu_{d_n}>0$ for all sufficiently large
$n$. Assume that there exists a cutoff sequence $(M_n)$ such that
$d_n=o(a_{M_n})$.
Set $
D_n
=
\#\{u\in\mathcal T^n:k_u(\mathcal T^n)=d_n\}$.
Then:
\begin{enumerate}
\item[(i)] If $(n\mu_{d_n})$ is bounded, then
$$
\dTV\left(D_n,\Po(n\mu_{d_n})\right)
\longrightarrow0.
$$
In particular, if $n\mu_{d_n}\to\lambda\in[0,\infty)$, then
$D_n$ converges in distribution to $\Po(\lambda)$.

\item[(ii)] If $n\mu_{d_n}\to\infty$, then
$$
\frac{D_n}{n\mu_{d_n}}
\quad\mathop{\longrightarrow}^{(\P)}_{n\to\infty}\quad
1.
$$
\end{enumerate}
\end{proposition}
\begin{proof}

We apply Theorem \ref{thm:PoissonApproximation}. Let $S_d$ be the star with root degree $d$, that is, the tree
consisting of one root and $d$ leaves. Then $S_d\prec T$ is equivalent to
$k_\varnothing(T)=d$. Thus the above count is the non-fringe count associated
with
$$
G_d(T)=\mathbbm{1}_{\{k_\varnothing(T)=d\}}
      =\mathbbm{1}_{\{S_d\prec T\}},
$$
and
$$
\P(G_{d_n}(\mathcal T)=1)=\mu_{d_n}.
$$
Conditionally on $G_d(\mathcal T)=1$, the tree $\mathcal T$ is obtained from
the star $S_d$ by grafting $d$ independent Bienaym\'e trees onto its leaves. Thus we may take $C_d=1$ and $L_d=d$. Since $M_n\to\infty$, for every
$\varepsilon>0$,
$$
a_n\P_{d_n}(C_{d_n}>\varepsilon M_n)=0
$$
for all sufficiently large $n$, and
$$
\frac{\E_{d_n}[L_{d_n}]}{a_{M_n}}
=
\frac{d_n}{a_{M_n}}
\longrightarrow0.
$$
Hence $(G_d)$ satisfies Assumption \ref{hyp:S} along $(d_n)$ with respect
to $(M_n)$, and Corollary \ref{cor:SM}\textup{(i)} gives Assumption
\ref{hyp:M}.

It remains only to check the local clustering condition. Let $A_n(\ell)$ be
the two-point quantity associated with $G_{d_n}$. For $\ell\ge1$,
conditioning on $k_\varnothing(\mathcal T)=d_n$, the vertex
$u_{\ell+1}(\mathcal T)$, if it exists, lies in the forest grafted on the
$d_n$ children of the root. Exposing this forest in depth-first order, the
event that this vertex has outdegree $d_n$ contributes a further factor
$\mu_{d_n}$. Hence, uniformly in $\ell\ge1$,
$$
A_n(\ell)
=
\P\bigl(
|\mathcal T|\ge \ell+1,\
k_\varnothing(\mathcal T)=d_n,\
k_{u_{\ell+1}(\mathcal T)}(\mathcal T)=d_n
\bigr)
\le
\mu_{d_n}^2.
$$
Therefore Assumption \ref{hyp:C}  holds since $M_n=o(n)$. The two conclusions then
follow directly from Theorem \ref{thm:PoissonApproximation}.
\end{proof}

 \section{Proof of the Poisson approximation using conditioned random walks}
\label{sec:proofapprox}

In this section we prove Theorem \ref{thm:PoissonApproximation}. We first
reduce the marked fringe subtree count to a sum of local indicators along a
conditioned random walk, and then apply the Chen--Stein method to these
truncated indicators. The required random-walk estimates are those collected
in Section \ref{sec:toolbox}.

As previously mentioned, we implicitly assume in the proofs that $\mu$ is aperiodic.

\subsection{Strategy of the proof of Theorem \ref{thm:PoissonApproximation}}

We briefly describe the structure of the proof of Theorem \ref{thm:PoissonApproximation}. The starting point is to use the framework of Section \ref{sec:indicators}, which gives the equality
$$\# \{u \in \mathcal{T}^{n} : G_{k_{n}}( \mathcal{T}^{n}_{u})=1\}=\sum_{i=1}^{n} I^{(n)}_{i}(\mathbf{x}( \mathcal{T}^{n})),$$
where the indicators $I_i^{(n)}$ are those of Definition
\ref{def:indicators} applied with $G=G_{k_n}$. An important tool will be Lemma \ref{lem:RW}, which shows that  the law of $\sum_{i=1}^{n} I^{(n)}_{i}(\mathbf{x}( \mathcal{T}^{n}))$ can be expressed using a conditioned random walk, so that \eqref{eq:DTVthm} can be reformulated as
\begin{equation}
\label{eq:DTVthm2}
\dTV \left(
\sum_{i=1}^{n} I^{(n)}_{i}(X_{1}, \ldots,X_{n})
\textrm{ under } \P(\cdot\mid S_n=-1),
\Po(n\pi_n)
\right)
\longrightarrow0.
\end{equation}

Let $M_n$ be a sequence such that  $M_n\to\infty$ and $M_n=o(n)$ and such that 
 Assumptions \ref{hyp:M} and \ref{hyp:C} are satisfied  for $(G_k)$ along $(k_{n})$ with respect to $(M_{n})$.
\begin{enumerate}
\item[--] \emph{Step 1.}  We show that
$$\dTV \left( \sum_{i=1}^{n} I^{(n)}_{i}(X_{1}, \ldots,X_{n}) \textrm{ under } \P( \cdot \mid S_{n}=-1), \sum_{i=1}^{n} I^{(M_{n})}_{i}(X_{1}, \ldots,X_{n}) \textrm{ under } \P( \cdot \mid S_{n}=-1)\right)$$
tends to $0$ as $n \rightarrow\infty$.
 In other words, without loss of generality we may only consider fringe subtrees of size at most $M_{n}$. This is the content of Lemma \ref{lem:truncate} (i).
\item[--] \emph{Step 2.}  We apply the Chen--Stein Method described below in Section \ref{sec:chen_stein_framework} to get
$$
\dTV \left( \sum_{i=1}^{n} I^{(M_{n})}_{i}(X_{1}, \ldots,X_{n}) \textrm{ under } \P( \cdot \mid S_{n}=-1),  \Po(n q_{n})\right)  \quad \mathop{\longrightarrow}_{n \rightarrow \infty} \quad0
$$
with $q_n= \P(I^{(M_{n})}_{1}(X_{1}, \ldots,X_{n}) =1  \mid S_{n}=-1)$.
\item[--] \emph{Step 3.} We show that
$
\dTV(\Po(nq_n),\Po(n\pi_n))\to0$,
where 
$\pi_n=\P(G_{k_n}(\mathcal T)=1)$. This is the content of Lemma
\ref{lem:truncate}\textup{(ii)}.
\end{enumerate}

\subsection{The Chen--Stein method}
\label{sec:chen_stein_framework}
Set $[n]= \{1,2, \ldots,n\}$. The following result immediately follows from the proof of Theorem 1 in \cite{AGG89} (by combining Lemma 1 with (12), in the notation of that paper).
\begin{proposition}[Arratia, Goldstein \& Gordon]
\label{prop:AGG}
Let $(I^{n}_i)_{i \in [n]}$ be indicator random variables with
$q^{n}_i = \E[I^{n}_i]$, $\lambda^{n} = \sum_{i=1}^n q^n_i$,
$W^n = \sum_{i=1}^n I^n_i$.
For each $i \in [n]$, let $B^n_i \subseteq [n]$ be such that $i \in B^n_i$. Then there exists a function $f_n$ with $||f_n||_{\infty} \leq 1$  such that
\[
  \dTV\bigl(W^n,\, \Po(\lambda^{n})\bigr)
  \leq  2 (b^n_1 + b^n_2 + b^n_3)
\]
with $$ 
b^n_1= \textstyle\sum_{i=1}^n \sum_{j \in B^n_i} q^n_i\, q^n_j,  \quad  b^n_2 = \textstyle\sum_{i=1}^n \sum_{j \in B^n_i \setminus \{i\}} \E[I^n_i I^n_j], \quad 
  b^n_3 = \textstyle\sum_{i=1}^n \left| \E\bigl[(I^n_i - q^n_i)\, f_n\big( 1 +\sum_{j\in [n]\setminus B^n_i} I^n_j\big)\bigr] \right|.
  $$
\end{proposition}

One often uses simpler bounds of $b^{n}_{3}$, but for the proof of Theorem \ref{thm:PoissonApproximation} it is important to use this more precise expression.

We apply this result in Step 2 described above with
$$B_i^n=\{j\in[n]: |i-j|_{\mathrm{cyc}}<M_n\} \quad \textrm{where} \quad |i-j|_{\mathrm{cyc}}=\min(|i-j|,n-|i-j|)$$
and show that $b^{n}_{1}, b^{n}_{2}, b^{n}_{3} \rightarrow0$. The proof of $b^{n}_{2} \rightarrow 0$ uses Assumption \ref{hyp:C}, while the arguments to show that $b^{n}_{1}, b^{n}_{3} \rightarrow0$ are based on quite general estimates that do not rely on the precise form of $G_{k}$ (Propositions \ref{prop:b1} and \ref{prop:b3bis}, respectively).

\subsection{Removing the bridge conditioning}

\begin{lemma}\label{lem:error_est}
Let $(m_{n})$ be a sequence of positive integers such that $m_{n}=o(n)$. Let $H : (\mathbb{Z}_{\geq -1})^{m_n} \to \{0,1\}$ be a function and let $(z_{n})$ be a sequence of positive real numbers such that $z_{n}=o(a_{n})$. Then there exists a constant~$C > 0$ (independent of $n$ and of $H$) such that for every $n \geq 1$ sufficiently large we have
\begin{equation}\label{eq:error_est}
  \sup_{|k| \leq z_{n}}\E\bigl[H(X_1,\ldots,X_{m_n}) \mid S_n = k\bigr] \leq C\, \E\bigl[H(X_1,\ldots,X_{m_n})\bigr].
\end{equation}
\end{lemma}

\begin{proof}
Set $\Psi_{n}(i)=\Pr{S_{n}=i}$ for $n \geq 1$ and $i \in \Z$. By the Markov property applied at time $m_{n}$, we have for every $k \in \Z$
$$
  \E\bigl[H(X_1,\ldots,X_{m_n}) \mid S_n = k\bigr]
  =
\E\left[H(X_1,\ldots,X_{m_n})   \frac{\Psi_{n-m_{n}}(k-S_{m_{n}})}{\Psi_{n}(k)}\right].
  $$
By the local limit theorem \eqref{eq:LL}, there exists a constant $c>0$ such that $\Psi_{n}(k) \geq c/a_{n}$ for every $n \geq 1$ and $|k| \leq z_{n}$. In addition, again by the local limit theorem  \eqref{eq:LL}, since $(a_{n})$ is regularly varying,  there exists a constant $c'>0$ such that for every $i \in \Z$ we have $\Psi_{n-m_{n}}(i) \leq c'/a_{n}$, and the desired result readily follows.
\end{proof}

The following result immediately follows.
\begin{corollary}
\label{cor:pip}
 Let $(M_{n})$ be a sequence of integers such that $M_n =o(n)$. There exists a constant $C>0$ such that for every $n \geq 1$ we have
$$\P(I^{(M_{n})}_{1}(X_{1}, \ldots,X_{n})=1 \mid S_{n}=-1) \leq  C\P(G_{k_{n}}( \mathcal{T})=1)=C\P(I^{(\infty)}_{1}(X_{1}, \ldots,X_{\zeta})=1),
$$
where the indicator $I_1^{(M)}$ is the one of Definition
\ref{def:indicators} applied with $G=G_{k_n}$.
\end{corollary}

 \subsection{Truncations}
 \label{ssec:truncate}
For $M, n \geq 1$ set
$$W^{M}_{n}= \sum_{i=1}^{n} I^{(M)}_{i}(X_{1}, \ldots,X_{n}),$$
  where the indicators $I_i^{(M)}$ are those of Definition
\ref{def:indicators} applied with $G=G_{k_n}$.

Recall the notation $\pi_n=\P(G_{k_n}(\mathcal T)=1)$ and set 
$$q_n=\P(I^{(M_n)}_1(X_1,\ldots,X_n)=1\mid S_n=-1).$$

\begin{lemma}
\label{lem:truncate}
Assume that \ref{hyp:M} holds for $(G_k)$ along $(k_{n})$ with respect to $(M_{n})$ and that $(n\P(G_{k_n}(\mathcal T)=1))$ is bounded. 
Then:
\begin{enumerate}
\item[(i)] We have
$
\dTV \left( W^{M_n}_{n} \textrm{ under } \P(\cdot\mid S_n=-1),
W^{n}_{n} \textrm{ under } \P(\cdot\mid S_n=-1)\right)\to0
$.
\item[(ii)] We have $nq_n-n\pi_n\to0$. In particular,
$$
\dTV(\Po(nq_n),\Po(n\pi_n))\to0.
$$
\end{enumerate}
\end{lemma}

\begin{proof}
We first prove \textup{(i)}. By Lemma \ref{lem:RW}, $W_n^n$ under $\P(\cdot\mid S_n=-1)$ has the same law as the number of vertices $u\in\mathcal T^n$ such that $G_{k_n}(\mathcal T^n_u)=1$, and $W_n^{M_n}$ has the same law as the same count restricted to vertices satisfying $|\mathcal T^n_u|\le M_n$. Since $0\le W_n^{M_n}\le W_n^n$, we have
$$\dTV(W_n^{M_n},W_n^n)\le\E[W_n^n-W_n^{M_n}\mid S_n=-1]=n\P(G_{k_n}(\mathcal T^n_{U_n})=1,\ |\mathcal T^n_{U_n}|>M_n),$$ where $U_n$ is a uniform vertex of $\mathcal T^n$, independent of $\mathcal T^n$. By Assumption \ref{hyp:M}, this last probability is $o(\pi_n)$. Since $(n\pi_n)$ is bounded, $\dTV(W_n^{M_n},W_n^n)=o(n\pi_n)=o(1)$. This proves \textup{(i)}.

For \textup{(ii)}, set $q_n=\P(I_1^{(M_n)}(X_1,\ldots,X_n)=1\mid S_n=-1)$. Again by Lemma \ref{lem:RW} and the uniform-vertex representation, $q_n=\P(G_{k_n}(\mathcal T^n_{U_n})=1,\ |\mathcal T^n_{U_n}|\le M_n)$. Therefore $$q_n-\pi_n=(\P(G_{k_n}(\mathcal T^n_{U_n})=1)-\pi_n)-\P(G_{k_n}(\mathcal T^n_{U_n})=1,\ |\mathcal T^n_{U_n}|>M_n).$$ By Assumption \ref{hyp:M}, both terms on the right-hand side are $o(\pi_n)$, and so $q_n=\pi_n+o(\pi_n)$. Since $(n\pi_n)$ is bounded, $nq_n-n\pi_n=o(n\pi_n)=o(1)$. Finally, $\dTV(\Po(nq_n),\Po(n\pi_n))\le |nq_n-n\pi_n|\to0$, which proves \textup{(ii)}.
\end{proof}

\subsection{Estimating the error terms in the Chen--Stein method}
\label{ssec:error}

For every sequence $\mathbf{x}$ and $n,M\ge1$, we denote by $I_i^{(M)}(\mathbf{x})$ the quantity introduced in Definition \ref{def:indicators}, applied to the $\{0,1\}$-valued function $G_{k_n}$. Throughout this subsection, we assume that \ref{hyp:M} holds for $(G_k)$ along $(k_n)$ with respect to $(M_n)$.

We shall apply Proposition \ref{prop:AGG} for the indicators $(I^{(M_{n})}_{i}(X_{1}, \ldots,X_{n}))_{1 \leq i \leq  n}$ under the conditional probability $\P(\cdot \mid S_{n}=-1)$ with
$$B_i^n=\{j\in[n]: |i-j|_{\mathrm{cyc}}<M_n\} \quad \textrm{where} \quad |i-j|_{\mathrm{cyc}}=\min(|i-j|,n-|i-j|),
$$
  and where the indicators $I_i^{(M_{n})}$ are those of Definition
\ref{def:indicators} applied with $G=G_{k_n}$.
We keep the notation 
\begin{equation}
\label{eq:defpi}
\pi_{n}=\P(G_{k_{n}}( \mathcal{T})=1)=\P(I^{(\infty)}_{1}(X_{1}, \ldots,X_{\zeta})=1)
\textrm{ and }  q_{n}=\P(I^{(M_{n})}_{1}(X_{1}, \ldots,X_{n})=1 \mid S_{n}=-1),
\end{equation}
where we recall that $\zeta=\zeta((X_{i})_{i \geq 1})$ is  the first hitting time of $-1$ by $(S_{k})_{k \geq 1}$.

To simplify notation, we shall write
$$I^{(M_{n})}_{i}=I^{(M_{n})}_{i}(X_{1}, \ldots,X_{n}),$$
which only depends on $X_{i}, \ldots,X_{i+M_{n}-1}$.
Also, since by cyclic invariance $\P(I^{(M_{n})}_{i}=1 \mid S_{n}=-1)$ does not depend on $i \in [n]$, we denote by $q_{n}$ this quantity. 

In this context, the error terms $b^{n}_{1}, b^{n}_{2}, b^{n }_{3}$ of Proposition \ref{prop:AGG} can be written as follows.
$$b_{1}^{n} = \sum_{i=1}^n \sum_{j \in B^n_i} q_{n}^{2}, \qquad  b_{2}^{n}=\sum_{i=1}^n \sum_{j \in B^n_i \setminus \{i\}} \E\left[I^{(M_{n})}_i I^{(M_{n})}_j \mid S_{n}=-1 \right]$$
and
\begin{eqnarray*}
b^{n}_{3}&=&\sum_{i=1}^n \left| \E\bigl[(I^{(M_{n})}_i - q_{n})\, f_{n}\big( 1 +\sum_{j\in [n]\setminus B^n_i} I^{(M_{n})}_j\big) \mid S_{n}=-1\bigr]\right|\\
&=& n\left| \E\bigl[(I^{(M_{n})}_1 - q_{n})\, f_{n}\big( 1 +\sum_{j=M_{n}+1}^{n-M_{n}+1} I^{(M_{n})}_j\big) \mid S_{n}=-1\bigr]\right|.\end{eqnarray*}
for a function $f_{n}$ with $||f_{n}||_{\infty} \leq 1$, where we have used exchangeability for the second equality.

We will need the following technical result.

\begin{lemma}
\label{lem:pnqn}
Assume \ref{hyp:M} holds for $(G_k)$ along $(k_{n})$ with respect to $(M_{n})$. Set $p_n = \P(I_1^{(M_n)}=1)$ and $r_n = \sqrt{a_{M_n} a_n}$. Then:
\begin{enumerate}
\item[(i)] $\P(I_1^{(M_n)} = 1, |S_{M_n}| > r_n) = o(p_n)$.
\item[(ii)] $p_n \sim q_n \sim \pi_n$.
\end{enumerate}
\end{lemma}
\begin{proof}
Let us first show that $p_n\sim q_n$.
Write $\Psi_m(i)=\P(S_m=i)$. By the Markov property at time $M_n$,
$$
q_n
=
\frac{1}{\Psi_n(-1)}
\E\left[
I_1^{(M_n)}
\Psi_{n-M_n}(-1-S_{M_n})
\right].
$$
Since $a_n$ is regularly varying with positive index $1/\alpha$ and $M_n=o(n)$, Potter bounds \eqref{eq: potter_bound} imply that $a_{M_n}/a_n\to0$. Hence $r_n/a_{M_n}\to\infty$ and $r_n/a_n\to0$. On the event $\{|S_{M_n}|\le r_n\}$, the local limit theorem gives $\Psi_{n-M_n}(-1-S_{M_n})/\Psi_n(-1)\to1$ uniformly. 
It remains to control the contribution of $\{|S_{M_n}|>r_n\}$. Since $I_1^{(M_n)}=0$ on $\{\zeta>M_n\}$ and, on $\{\zeta=j\}$, the event $\{I_1^{(M_n)}=1\}$ is measurable with respect to $(X_1,\ldots,X_j)$, the Markov property at time $\zeta$ gives
$$
\P(I_1^{(M_n)}=1,\ |S_{M_n}|>r_n)
\le
\sum_{j=1}^{M_n}
\E\left[I_1^{(M_n)}\mathbbm{1}_{\{\zeta=j\}}\right]
\P(|S_{M_n-j}|>r_n-1).
$$
Indeed,
$$
\max_{0\le j\le M_n}
\P(|S_{M_n-j}|>r_n-1)
\le
\P\left(\max_{0\le i\le M_n}|S_i|>r_n-1\right)=o(1),
$$
by tightness of
$(S_{\lfloor M_n t\rfloor}/a_{M_n}:0\le t\le1)$ and
$r_n/a_{M_n}\to\infty$. Since
$\sum_{j=1}^{M_n}\E[I_1^{(M_n)}\mathbbm{1}_{\{\zeta=j\}}]=p_n$,
the preceding display is $o(p_n)$. This establishes \textup{(i)}.
Moreover, the local limit theorem gives
$\Psi_{n-M_n}(\cdot)/\Psi_n(-1)\le C$ uniformly. Hence the contribution of
$\{|S_{M_n}|>r_n\}$ to the expression for $q_n$ is $o(p_n)$, and therefore
$q_n=p_n+o(p_n)$.

Let $U_n$ be a uniform vertex of $\mathcal T^n$, independent of
$\mathcal T^n$. By the cyclic-shift representation,
$$
q_n=\P\bigl(
G_{k_n}(\mathcal T^n_{U_n})=1,\
|\mathcal T^n_{U_n}|\le M_n
\bigr).
$$
Assumption \ref{hyp:M} therefore gives $q_n\sim\pi_n$.
Together with $p_n\sim q_n$, this proves \textup{(ii)}.
\end{proof}

\subsubsection{Bounding $b^{n}_1$}
\label{sec:b1_common}

\begin{proposition}\label{prop:b1} Assume that  $(n \pi_{n})$ is bounded. Then $b_{1}^{n} \rightarrow 0$.
\end{proposition}

\begin{proof}
Using Corollary \ref{cor:pip}, write
  $$
  \sum_{i=1}^n \sum_{j \in B^n_i} q_{n}^{2}
  \leq n \cdot 2M_n \cdot q_n^2  \leq Cn \cdot M_n \cdot \pi_n^2   = \Obig\!\left(\frac{M_n}{n}\right)  \to 0,
$$
since $M_n = o(n)$, which is the desired result.
\end{proof}

\subsubsection{Bounding $b^{n}_2$}
\label{sec:b2_common}

\begin{lemma}
\label{lem:boundb2}
Assume \ref{hyp:C}  for $(G_{k})$ along $(k_{n})$ with respect to $(M_{n})$ and that $(n \pi_{n})$ is bounded. Then $b^{n}_{2}  \rightarrow  0$.
\end{lemma}

\begin{proof}
By cyclic invariance and Lemma \ref{lem:error_est}, applied to the product
$I_1^{(M_n)}I_{1+\ell}^{(M_n)}$, which depends on at most $2M_n=o(n)$
increments for $1\leq \ell\leq M_n-1$, there exists $C>0$ such that
$$
b_2^n
\leq
C n\sum_{\ell=1}^{M_n-1}
\E\left[I_1^{(M_n)}I_{1+\ell}^{(M_n)}\right].
$$
It remains to bound the unconditioned expectation. Let $I_i^{(\infty)}$ denote the corresponding untruncated indicator. Since $I_i^{(M_n)}\leq I_i^{(\infty)}$, it is enough to estimate
$\P(I_1^{(\infty)}=1,\ I_{1+\ell}^{(\infty)}=1)$.

Consider the infinite forest encoded by $(X_i)_{i\ge1}$, and let
$(v_i)_{i\ge1}$ be its vertices listed in depth-first order. Let
$\mathcal T$ be the first tree of this forest, rooted at $v_1$. Fix $\ell\ge1$. If $v_{1+\ell}$ does not belong to $\mathcal T$, then the occurrence at $v_1$ and the occurrence at $v_{1+\ell}$ lie in two independent Bienaymé trees. This contribution is at most $\pi_n^2$. If $v_{1+\ell}$ belongs to $\mathcal T$, then $v_{1+\ell}=u_{\ell+1}(\mathcal T)$, and the event
$\{I_1^{(\infty)}=1,\ I_{1+\ell}^{(\infty)}=1\}$ implies
$|\mathcal T|\ge \ell+1$, $G_{k_n}(\mathcal T)=1$, and
$G_{k_n}(\mathcal T_{u_{\ell+1}(\mathcal T)})=1$. Thus, by definition of $A_n(\ell)$,
$$
\E\left[I_1^{(M_n)}I_{1+\ell}^{(M_n)}\right]
\leq
\pi_n^2+A_n(\ell).
$$
Therefore, using Assumption \ref{hyp:C},
$$
b_2^n
\leq
C \left(n M_n\pi_n^2+n\sum_{\ell=1}^{M_{n}}A_n(\ell)\right)
=
C n M_n\pi_n^2
+
\,o\left({n}^{2}\pi_n^2\right).
$$
Since $M_n=o(n)$ and $(n\pi_n)$ is bounded, both terms go to $0$ as $n \rightarrow\infty$. Hence $b_2^n\to0$.
\end{proof}

\subsubsection{Bounding $b^{n}_3$}\label{sec:b3_common}

\begin{proposition}\label{prop:b3bis}
Let $f_{n}$ be a function with $||f_{n}||_{\infty} \leq 1$. Assume that $(n \pi_{n})$ is bounded. Then $b^{n}_{3} \rightarrow 0$.
\end{proposition}

The proof is based on the following  lemma.

\begin{lemma}\label{lem:1} For $k \in \Z$ set
$$  h_{n}(k)= \E\!\left[ f_{n}\big( 1 +\sum_{j=1}^{n-2M_{n}+1} I^{(M_{n})}_j\big) \;\Big|\; S_{n-M_{n}} = -1-k\right], \quad \Delta_{n}(k) = h_{n}(k) - h_{n}(0).$$
Let $f_{n}$ be a function with $||f_{n}||_{\infty} \leq 1$. Assume that $(n \pi_{n})$ is bounded. Then
$$n\,\E\!\left[(I_1^{(M_n)} - q_{n})\,\Delta_{n}(S_{M_n}) \mid S_n = -1\right]  \quad \mathop{\longrightarrow}_{n \rightarrow \infty} \quad0.$$
\end{lemma}
Observe that since $I^{(M_{n})}_j$ only depends on $X_{j},X_{j+1}, \ldots,X_{j+M_{n}-1}$, the random variable in the argument of the function $f_{n}$ only depends on $(X_{1}, \ldots,X_{n-M_{n}})$.

Taking Lemma \ref{lem:1} for granted, the proof of Proposition \ref{prop:b3bis} is short.
\begin{proof}[Proof of Proposition \ref{prop:b3bis}]
Conditionally given $\{S_{M_n} = k,\, S_n = -1\}$, the blocks $(X_1,\ldots,X_{M_n})$ and $(X_{M_n+1},\ldots,X_n)$ are independent (being disjoint segments of i.i.d.\ variables conditioned on their respective sums to be $k$ and $-1-k$).  Thus
$$b^{n}_{3} =  n \left| \E\!\left[(I_1^{(M_n)} - q_n)\,h_{n}(S_{M_n}) \mid S_n = -1\right]\right| .$$
Since $\E[I_1^{(M_n)} \mid S_{n}=-1]=q_{n}$, it follows that
$$b^{n}_{3} =  n \left| \E\!\left[(I_1^{(M_n)} - q_n)\,\Delta_{n}(S_{M_n}) \mid S_n = -1\right] \right|.$$
and the desired result follows from Lemma \ref{lem:1}.
\end{proof}

It remains to establish Lemma \ref{lem:1}.
\begin{proof}[Proof of Lemma \ref{lem:1}] We keep the notation $\Psi_{n}(i)=\Pr{S_{n}=i}$ and set 
$$R_{n} = n\,\E\!\left[(I_1^{(M_n)} - q_n)\,\Delta_{n}(S_{M_n}) \mid S_n = -1\right].$$
Recall that $\zeta=\zeta((X_{i})_{i \geq 1})$ is  the first hitting time of $-1$ by $(S_{k})_{k \geq 1}$.

 We first truncate according to the value of $S_{M_{n}}$. Set $r_n=\sqrt{a_{M_n}a_n}$. Since $a_n$ is regularly varying with positive
index $1/\alpha$ and $M_n=o(n)$, Potter bounds  \eqref{eq: potter_bound} imply that
$a_{M_n}/a_n\to0$. Hence $r_n/a_{M_n}\to\infty$ and $r_n/a_n\to0$.

Then $|R_n| \leq R_n^1 + R_n^2$ where

\begin{align*}
  R_n^1 &= n\,\E\!\left[(I_1^{(M_n)} + q_{n})\,|\Delta_{n}(S_{M_n})|\,\mathbf{1}_{|S_{M_n}|\leq r_n} \mid S_n=-1\right], \\
  R_n^2 &= n\,\E\!\left[(I_1^{(M_n)} + q_{n})\,|\Delta_{n}(S_{M_n})|\,\mathbf{1}_{|S_{M_n}|> r_n} \mid S_n=-1\right].
\end{align*}

\paragraph{Bounding $R_{n}^{2}$.}
Since $|\Delta_{n}(k)| \leq 2\|f_{n}\|_\infty \leq 2$, we have
\[
  R_n^2 \leq 2n\,\E\!\left[I_1^{(M_n)}\,\mathbf{1}_{|S_{M_n}|>r_n} \mid S_n=-1\right] + 2n q_{n}\,\P(|S_{M_n}|>r_n \mid S_n=-1).
\]
For the second term, by Lemma \ref{lem:error_est}, $\P(|S_{M_n}|>r_n \mid S_n=-1) \leq C \P(|S_{M_n}|>r_n) \rightarrow 0$ since $S_{M_{n}}/a_{M_{n}}$ converges in distribution and $r_{n}/a_{M_{n}} \rightarrow\infty$.

For the first term, since $I_1^{(M_n)}$ is measurable with respect to $(X_1, \ldots, X_{M_n})$, Lemma \ref{lem:error_est} and Lemma \ref{lem:pnqn}\textup{(i)} give
$$ \E\!\left[I_1^{(M_n)}\,\mathbf{1}_{|S_{M_n}|>r_n} \mid S_n=-1\right] \le C \P(I_1^{(M_n)}=1, |S_{M_n}|>r_n) = o(p_n) = o(q_n). $$
Thus $n\,\E\!\left[I_1^{(M_n)}\,\mathbf{1}_{|S_{M_n}|>r_n} \mid S_n=-1\right] = o(nq_n) \to 0$ since $(nq_n)$ is bounded.

\paragraph{Bounding $R_n^1$.}
Since $\E[(I_1^{(M_n)}+q _{n})\mid S_n=-1] = 2 q_{n}$, we have
\[
  R_n^1 \leq 2n q_{n}\,\sup_{|k|\leq r_n}|\Delta_{n}(k)|.
\]
It is thus enough to show that
\begin{equation}
\label{eq:delta_vanish}\sup_{|k|\leq r_n}|\Delta_{n}(k)|  \quad \mathop{\longrightarrow}_{n \rightarrow \infty} \quad0.
\end{equation}
The proof is also based on a cutoff argument. For $\varepsilon \in (0,1)$, $n \geq 1$ and $k \in \Z$ set
$$
  h_n^{\varepsilon}(k) = \E\!\left[ f_{n}\big( 1 +\sum_{j=1}^{n-2M_{n}+1 - \lfloor \varepsilon n \rfloor} I^{(M_{n})}_j\big)  \;\Big|\; S_{n-M_{n}} = -1-k\right].$$
  We shall establish the following two estimates
  \begin{eqnarray}
\exists C>0,  \forall \varepsilon \in (0,1), & \quad \limsup_{n\to\infty}\sup_{|k|\leq r_n}|h_n^{\varepsilon}(k)-h_n(k)| \leq C\varepsilon \label{eq:approx1}\\
 \forall \varepsilon \in (0,1),&\quad  \lim_{n \rightarrow\infty} \sup_{|k|\leq r_n}|h_n^{\varepsilon}(k)-h_n^{\varepsilon}(0)|= 0   \label{eq:approx3}
\end{eqnarray}
By combining these two estimates and by writing
$$|h_n(k)-h_n(0)|  \leq  |h_n(k)-h_n^{\varepsilon}(k)| +  |h_n^{\varepsilon}(k)-h_n^{\varepsilon}(0)|  + |h_n^{\varepsilon}(0)-h_n(0)|,$$
\eqref{eq:delta_vanish} readily follows.

\emph{Proof of \eqref{eq:approx1}.}  Observe that the random variables in the argument of $f_{n}$ in $h_n^{\varepsilon}(k)$ and  $h_n(k)$ are equal unless  $\sum_{i=n-2M_{n}+1-\lfloor \varepsilon n \rfloor}^{n-2M_{n}+1} I_i^{(M_{n})} \geq 1$. Thus for $k \in \Z$
\begin{eqnarray}
 |h_n^\varepsilon(k)-h_n(k)| &\leq& 2 \P\!\left(\sum_{i=n-2M_{n}+1-\lfloor \varepsilon n \rfloor}^{n-2M_{n}+1} I_i^{(M_{n})}  \geq 1\;\Big|\; S_{n-M_n}=-1-k\right) \notag \\
 & \leq& 2( \lfloor \varepsilon n \rfloor +1) \,\P\left( I_1^{(M_{n})}=1\mid S_{n-M_{n}}=-1-k\right) \label{eq:hne}
\end{eqnarray}
where the second inequality follows from a union bound combined with exchangeability. 

It remains to estimate
$\P(I_1^{(M_n)}=1\mid S_{n-M_n}=-1-k)$. Apply Lemma
\ref{lem:error_est} with total length $N_n=n-M_n$, local block length
$m_n=M_n$, and $z_n=1+r_n$. Indeed, $M_n=o(N_n)$ and
$1+r_n=o(a_{N_n})$, since $a_{n-M_n}\sim a_n$ and $r_n=o(a_n)$.
Uniformly for $|k|\le r_n$, Lemma \ref{lem:error_est} and Lemma
\ref{lem:pnqn}\textup{(ii)} therefore give
$$
\P(I_1^{(M_n)}=1\mid S_{n-M_n}=-1-k)
\le C\P(I_1^{(M_n)}=1)=Cp_n\le C'q_n.
$$
Since $nq_n=O(1)$, substituting this estimate into \eqref{eq:hne} proves
\eqref{eq:approx1}.

\emph{Proof of \eqref{eq:approx3}.} For $n \geq 1$ and $k \in \Z$ set
$$X^{\varepsilon}_{n}(k)=\left|\frac{\Psi_{\lfloor \varepsilon n \rfloor}(-1-k-S_{n-M_{n}-\lfloor \varepsilon n \rfloor})}{\Psi_{n-M_{n}}(-1-k)} - \frac{\Psi_{\lfloor \varepsilon n \rfloor}(-1-S_{n-M_{n}-\lfloor \varepsilon n \rfloor})}{\Psi_{n-M_{n}}(-1)}\right|.
$$
Using the Markov property at time $n-M_{n}-\lfloor \varepsilon n \rfloor$, write
$$  \sup_{|k|\leq r_n}|h_n^{\varepsilon}(k)-h_n^{\varepsilon}(0)|
\leq 
\E\!\left[ \sup_{|k|\leq r_n} X^{\varepsilon}_{n}(k) \right].
$$

Fix $\varepsilon\in(0,1)$. The local limit theorem and
$r_n=o(a_n)$ yield, uniformly in $|k|\le r_n$,
$$
  \Psi_{n-M_n}(-1-k)\ge \frac{c}{a_n}
  \qquad\text{and}\qquad
  \sup_x\Psi_{\lfloor\varepsilon n\rfloor}(x)
  \le \frac{C}{a_{\lfloor\varepsilon n\rfloor}}.
$$
Since $a_n/a_{\lfloor\varepsilon n\rfloor}=O_\varepsilon(1)$, we obtain
$\sup_{|k|\le r_n}X_n^\varepsilon(k)\le C_\varepsilon$ for all large $n$. In addition, since $S_{n-M_{n}-\lfloor \varepsilon n \rfloor}/a_{n}$ converges in distribution as $ n \rightarrow\infty$, without loss of generality by Skorokhod's representation theorem we may assume that this convergence holds almost surely. Then the  local limit theorem, the uniform continuity of the
stable density, and $r_n/a_n\to0$ imply  that $
  \sup_{|k|\le r_n}X_n^\varepsilon(k)$ converges almost surely to $0$.
Dominated convergence therefore gives
$
  \E[\sup_{|k|\le r_n}X_n^\varepsilon(k)]
\to 0$,
which proves \eqref{eq:approx3}.
\end{proof}

\subsection{Boundary cases}

We keep the notation of Section \ref{ssec:error} and in particular write $I^{(M_{n})}_{i}=I^{(M_{n})}_{i}(X_{1}, \ldots,X_{n})$,   where the indicators $I_i^{(M_{n})}$ are still those of Definition
\ref{def:indicators} applied with $G=G_{k_n}$. It will be useful to have bounds in the two boundary cases
$n \pi_{n} \rightarrow 0$ and $n \pi_{n} \rightarrow \infty$.

\begin{lemma}
\label{lem:extreme_means}
Assume \ref{hyp:M}  for $(G_{k})$  along $(k_{n})$ with respect to $(M_{n})$. The following statements hold. 
\begin{enumerate}[label=\textup{(\roman*)}]
\item  If $n\pi_n\to0$, then under $\P(\cdot\mid S_n=-1)$,
$\sum_{i=1}^n I_i^{(M_n)}$ converges in probability to $0$;
\item  Assume in addition \ref{hyp:C} for $(G_k)$  along $(k_{n})$ with respect to $(M_{n})$. If $n\pi_n\to\infty$, then $$ \frac{\sum_{i=1}^n I_i^{(M_n)}}{n\pi_n}  \quad \mathop{\longrightarrow}^{(\P)}_{n \rightarrow \infty} \quad1 \qquad \textrm{under } \P(\cdot\mid S_n=-1).$$
\end{enumerate}
\end{lemma}

\begin{proof}
Set $W_n^{M_n}=\sum_{i=1}^n I_i^{(M_n)}$. We keep the notation $ \pi_{n}=\P(G_{k_{n}}( \mathcal{T})=1)$, $q_n=\P(I_1^{(M_n)}=1\mid S_n=-1)$ and $p_n=\P(I_1^{(M_n)}=1)$. By Lemma \ref{lem:pnqn}, $p_n \sim q_n \sim \pi_n$.

If $n\pi_n\to0$, then
$\E[W_n^{M_n}\mid S_n=-1]=nq_n\sim n\pi_n\to0$, and Markov's inequality gives
$W_n^{M_n}\to0$ in probability.

Assume now that $n\pi_n\to\infty$. Put
$\mu_n=\E[W_n^{M_n}\mid S_n=-1]=nq_n$. Then
$\mu_n\sim n\pi_n\to\infty$. We prove the stronger statement
$W_n^{M_n}/\mu_n\to1$ in probability under $\P(\cdot\mid S_n=-1)$.

We first estimate the variance. The diagonal contribution to
$\E[(W_n^{M_n})^2\mid S_n=-1]$ is $\mu_n=o(\mu_n^2)$. For the local pairs, the same argument as in the proof of Lemma
\ref{lem:boundb2} gives
$$
\sum_{i=1}^n
\sum_{1\le |j-i|_{\mathrm{cyc}}<M_n}
\E[I_i^{(M_n)}I_j^{(M_n)}\mid S_n=-1]
\le
C n\left(M_n\pi_n^2+\sum_{\ell=1}^{M_n-1}A_n(\ell)\right).
$$
Assumption \ref{hyp:C} gives
$\sum_{\ell=1}^{M_n}A_n(\ell)=o(n\pi_n^2)$. Since $M_n=o(n)$, the local-pair
contribution is $o(n^2\pi_n^2)=o(\mu_n^2)$.

It remains to control the separated pairs. We claim that, uniformly for
$M_n+1\le j\le n-M_n+1$,
$$
\P(I_1^{(M_n)}=1,\ I_j^{(M_n)}=1\mid S_n=-1)
=
q_n^2+o(\pi_n^2).
$$
By cyclic exchangeability, this estimate applies to every separated pair. Since
the number of non-separated pairs is $O(nM_n)$, the separated contribution is
$(n^2+O(nM_n))q_n^2+o(n^2\pi_n^2)=\mu_n^2+o(\mu_n^2)$, because
$M_n/n\to0$ and $q_n\sim\pi_n$.

We now prove the claim. Set $Z_{j,n}=S_{j+M_n-1}-S_{j-1}$. Since the blocks
$(X_1,\ldots,X_{M_n})$ and $(X_j,\ldots,X_{j+M_n-1})$ are disjoint, they are
independent under the unconditioned law. Hence the Markov property gives
$$
\P(I_1^{(M_n)}=1,\ I_j^{(M_n)}=1\mid S_n=-1)
=
\frac{1}{\Psi_n(-1)}
\E\left[
I_1^{(M_n)}I_j^{(M_n)}
\Psi_{n-2M_n}(-1-S_{M_n}-Z_{j,n})
\right].
$$
On $\{|S_{M_n}|\le r_n,\ |Z_{j,n}|\le r_n\}$, the local limit theorem gives
$\Psi_{n-2M_n}(-1-S_{M_n}-Z_{j,n})/\Psi_n(-1)\to1$ uniformly. On the
complement, the local limit theorem gives
$\Psi_{n-2M_n}(\cdot)/\Psi_n(-1)\le C$, and independence of the two blocks
gives
$$
\P(I_1^{(M_n)}I_j^{(M_n)}\mathbbm{1}_{\{|S_{M_n}|>r_n\}})
\le
\P(I_1^{(M_n)}=1,\ |S_{M_n}|>r_n)\,p_n
=o(p_n^2),
$$
where the last equality follows from Lemma \ref{lem:pnqn}\textup{(i)}, and the same bound holds when $Z_{j,n}$ replaces $S_{M_n}$. Therefore the
conditional probability above is $p_n^2+o(p_n^2)$. Since
$p_n\sim q_n\sim\pi_n$, the claim follows.

Combining the diagonal, local-pair and separated-pair estimates gives
$\Var(W_n^{M_n}\mid S_n=-1)=o(\mu_n^2)$. By Chebyshev's inequality,
$W_n^{M_n}/\mu_n\to1$ in probability under $\P(\cdot\mid S_n=-1)$. Since
$\mu_n=nq_n$ and $q_n\sim\pi_n$, this gives
$
{W_n^{M_n}}/(n\pi_n)\to1
$
in probability under $\P(\cdot\mid S_n=-1)$.
\end{proof}

\subsection{Proof of Theorem \ref{thm:PoissonApproximation}}

We now have all the ingredients to establish  Theorem \ref{thm:PoissonApproximation}.

\begin{proof}[Proof of Theorem \ref{thm:PoissonApproximation}]
We keep the notation $\pi_n=\P(G_{k_n}(\mathcal T)=1)$ and
$q_n=\P(I_1^{(M_n)}=1\mid S_n=-1)$. Let
$N_n=\#\{u\in\mathcal T^n:G_{k_n}(\mathcal T^n_u)=1\}$.
By Lemma \ref{lem:RW}, $N_n$ has the same law as
$W_n^n=\sum_{i=1}^n I_i^{(n)}$ under $\P(\cdot\mid S_n=-1)$.

Assume first that $(n\pi_n)$ is bounded. By Lemma \ref{lem:truncate},
$W_n^n$ and $W_n^{M_n}$ have total variation distance tending to $0$ under
$\P(\cdot\mid S_n=-1)$, and $nq_n-n\pi_n\to0$. We apply Proposition
\ref{prop:AGG} to the indicators $(I_i^{(M_n)})_{1\le i\le n}$ under
$\P(\cdot\mid S_n=-1)$, with neighborhoods
$B_i^n=\{j\in[n]: |i-j|_{\mathrm{cyc}}<M_n\}$. Since $(n\pi_n)$ is bounded,
Proposition \ref{prop:b1}, Lemma \ref{lem:boundb2} and Proposition
\ref{prop:b3bis} give $b_1^n\to0$, $b_2^n\to0$ and $b_3^n\to0$. Therefore
$$
\dTV(W_n^{M_n} \textrm{ under } \P( \cdot \mid S_{n}=-1),\Po(nq_n))\to0.
$$
Since $\dTV(\Po(nq_n),\Po(n\pi_n))\le |nq_n-n\pi_n|\to0$, the triangle
inequality gives
$$
\dTV(W_n^n \textrm{ under } \P( \cdot \mid S_{n}=-1),\Po(n\pi_n))\to0.
$$
Lemma \ref{lem:RW} then gives the desired
Poisson approximation for the tree count.

Assume now  that $n\pi_n\to\infty$.
By Lemma \ref{lem:extreme_means}\textup{(ii)},
$W_n^{M_n}/(n\pi_n)\to1$ in probability under $\P(\cdot\mid S_n=-1)$. Moreover,
by Assumption \ref{hyp:M},
$$
\E[W_n^n-W_n^{M_n}\mid S_n=-1]
=
n\P(G_{k_n}(\mathcal T^n_{U_n})=1,\ |\mathcal T^n_{U_n}|>M_n)
=
o(n\pi_n).
$$
Hence Markov's inequality gives
$(W_n^n-W_n^{M_n})/(n\pi_n)\to0$ in probability under
$\P(\cdot\mid S_n=-1)$. Therefore $W_n^n/(n\pi_n)\to1$ 
in probability under $\P(\cdot\mid S_n=-1)$. By Lemma \ref{lem:RW}, $
{N_n}/({n\pi_n})\to1$
in probability. In particular, $N_n\to\infty$ in probability. This completes the proof.
\end{proof}

\section{Declumping}
\label{sec:declumping}

This section proves the declumping result stated in Theorem
\ref{thm:PoissonApproximation2}. The idea is to replace raw marked
occurrences, which may form clusters, by indicators which select boundary
points of such clusters. The proof has two ingredients. First, Assumption
\ref{hyp:S} gives a structural description of a tree conditioned on a raw
marked event: the marked event can be witnessed by a small skeleton with a
small number of open leaves, onto which independent Bienaym\'e trees are
grafted. We show that this implies the microscopicity estimates needed to
apply Theorem \ref{thm:PoissonApproximation}.

We then apply Theorem \ref{thm:PoissonApproximation} to the corresponding
declumped indicators. Throughout this section, let $(G_k)_{k\ge1}$ be a
sequence of $\{0,1\}$-valued functions on finite plane trees, and fix a
sequence $(k_n)$ of integers. Let $B\subset\mathbb Z_+$ be such that
$$
  c_B
  =
  \E\left[
    k_\varnothing(\mathcal T)
    \mathbbm{1}_{\{k_\varnothing(\mathcal T)\in B\}}
  \right]
  \in(0,\infty).
$$
For every $k\ge1$, define
$$
  \widehat G_k(T)
  =
  \mathbbm{1}_{\{k_\varnothing(T)\in B\}}
  \mathbbm{1}_{\{\exists u\in\Gamma_\varnothing(T):G_k(T_u)=1\}}.
$$
Thus $\widehat G_k(T)=1$ means that the root of $T$ satisfies the degree
condition encoded by $B$ and has at least one child whose fringe subtree is
marked by $G_k$.

The root-declumping estimate below shows that, in the bounded-mean regime,
$\P(\widehat G_{k_n}(\mathcal T)=1)$ is asymptotic to
$c_B\P(G_{k_n}(\mathcal T)=1)$. Combined with the microscopicity estimates
deduced from Assumption \ref{hyp:S}, this proves Theorem
\ref{thm:PoissonApproximation2}. Finally, we record a deterministic comparison
lemma which allows us, in the extremal applications, to pass from declumped
counts back to the absence of raw occurrences.
 
 As previously mentioned, we implicitly assume in the proofs that $\mu$ is aperiodic.

\begin{lemma}
\label{lem:root_declumping}
Set $\pi_n=\P(G_{k_n}(\mathcal T)=1)$ and
$\widehat{\pi}_n=\P(\widehat{G}_{k_n}(\mathcal T)=1)$.
Set
$$
b_B=\P\bigl(k_\varnothing(\mathcal T)\in B,\
k_\varnothing(\mathcal T)\ge1\bigr)>0.
$$
Then $\widehat\pi_n\ge b_B\pi_n$ for every $n \geq 1$. In addition:
\begin{enumerate}
\item[(i)]  If $(n\pi_n)$ is bounded, then $
n\widehat{\pi}_n-c_B n\pi_n\to0$.
\item[(ii)] If $n\pi_n\to\infty$, then $n\widehat{\pi}_n\to\infty$.
\item[(iii)] Finally, if $(G_k)$ satisfies Assumption \ref{hyp:S} along $(k_n)$ with
respect to $(M_n)$, then $(\widehat G_k)$ satisfies Assumption \ref{hyp:M}
along $(k_n)$ with respect to the same cutoff.
\end{enumerate}
\end{lemma}
  
\begin{proof}
Set $D=k_\varnothing(\mathcal T)$. By the branching property,
\begin{equation}
\label{eq:parent_formula}
\widehat\pi_n
=\E\left[\mathbbm{1}_{\{D\in B\}}
\left(1-(1-\pi_n)^D\right)\right].
\end{equation}
Since $p\le1-(1-p)^d\le dp$ for $d\ge1$,
\eqref{eq:parent_formula} gives
$b_B\pi_n\le\widehat\pi_n\le c_B\pi_n$. If $\pi_n\to0$, dominated
convergence in \eqref{eq:parent_formula} gives $\widehat\pi_n/\pi_n\to c_B$. This proves both
claims (i) and (ii) concerning the mean.

It remains to prove the last assertion. Write
$\P_n=\P(\,\cdot\mid G_{k_n}(\mathcal T)=1)$ and let
$\mathcal T_1,\ldots,\mathcal T_D$ be the child subtrees of the root.
For a fixed child $i$, let $\mathcal F_{D-1}^{(i)}$ be the forest formed by
the other child subtrees. For all large $n$, on
$\{G_{k_n}(\mathcal T_i)=1,\ |\mathcal T|>M_n\}$, either
$|\mathcal T_i|>M_n/4$ or
$|\mathcal F_{D-1}^{(i)}|>M_n/2$. Hence a union bound and the first estimate
of Lemma \ref{lem:forest_estimates} give
\begin{equation}
\frac{\P(\widehat G_{k_n}(\mathcal T)=1,\ |\mathcal T|>M_n)}{\pi_n}\le c_B\P_n(|\mathcal T|>M_n/4)
+C\E\left[
D\mathbbm{1}_{\{D\in B\}}
\left(1\wedge\frac{D}{a_{M_n}}\right)\right]=o(1).
\label{eq:parent_unweighted}
\end{equation}
The first term tends to zero by Lemma \ref{lem:consequenceM}, and the second
one by dominated convergence.

Conditionally on $\mathcal T_i=t$ and $D=d$, the Markov property gives
\begin{equation}
\label{eq:parent_convolution}
\E\left[\Psi_{n-|\mathcal T|}(0)\mid\mathcal T_i=t,D=d\right]
\le \sup_x\Psi_{n-1-|t|}(x).
\end{equation}
Moreover, uniformly for $|t|\le M_n/4$, Lemma
\ref{lem:forest_estimates}, the local limit theorem, and
$n-1-|t|\sim n$ give
\begin{equation}
\label{eq:parent_forest}
\frac{1}{\Psi_n(-1)}
\E\left[
\mathbbm{1}_{\{|\mathcal F_{d-1}|>M_n/2\}}
\Psi_{n-1-|t|-|\mathcal F_{d-1}|}(0)
\right]
\le C\left(1\wedge\frac{d}{a_{M_n}}\right).
\end{equation}
The truncation by $1$ follows by combining the forest estimate with
the unconditional convolution bound in \eqref{eq:parent_convolution}.
Summing \eqref{eq:parent_convolution}--\eqref{eq:parent_forest} over
the possible marked children yields
\begin{align}
&\frac{1}{\pi_n\Psi_n(-1)}
\E\left[
\mathbbm{1}_{\{\widehat G_{k_n}(\mathcal T)=1,\ |\mathcal T|>M_n\}}
\Psi_{n-|\mathcal T|}(0)\right]\notag\\
&\quad\le
c_B\E_n\left[
\frac{\sup_x\Psi_{n-1-|\mathcal T|}(x)}{\Psi_n(-1)}
\mathbbm{1}_{\{|\mathcal T|>M_n/4\}}\right]+C\E\left[
D\mathbbm{1}_{\{D\in B\}}
\left(1\wedge\frac{D}{a_{M_n}}\right)\right]
=o(1).
\label{eq:parent_bridge}
\end{align}
Lemma \ref{lem:consequenceM} applies to the first term since replacing
$n$ by $n-1$ does not change its estimate.

Since $\widehat\pi_n\ge b_B\pi_n$, \eqref{eq:parent_unweighted} and
\eqref{eq:parent_bridge} are respectively
$o(\widehat\pi_n)$ before and after bridge weighting. Finally, uniformly on
$\{|\mathcal T|\le M_n\}$,
$$
\frac{\Psi_{n-|\mathcal T|}(0)}{\Psi_n(-1)}=1+o(1).
$$
Corollary \ref{cor:cyclicshift}, together with
\eqref{eq:parent_unweighted}--\eqref{eq:parent_bridge}, therefore gives
both conditions in Assumption \ref{hyp:M} for $(\widehat G_k)$.
\end{proof}

We are finally in position to establish Theorem \ref{thm:PoissonApproximation2}.

\begin{proof}[Proof of Theorem \ref{thm:PoissonApproximation2}]
Lemma \ref{lem:root_declumping} shows that $(\widehat G_k)$ satisfies
Assumption \ref{hyp:M} along $(k_n)$ with respect to $(M_n)$. By hypothesis
it also satisfies Assumption \ref{hyp:C}.

If $(n\pi_n)$ is bounded, Theorem
\ref{thm:PoissonApproximation}\textup{(i)}, applied to $\widehat G_k$, gives
$$
\dTV\left(
\#\{u\in\mathcal T^n:\widehat G_{k_n}(\mathcal T_u^n)=1\},
\Po(n\widehat\pi_n)\right)\longrightarrow0.
$$
Lemma \ref{lem:root_declumping} and
$\dTV(\Po(a),\Po(b))\le|a-b|$ allow us to replace
$n\widehat\pi_n$ by $c_Bn\pi_n$.

If $n\pi_n\to\infty$, the same lemma gives
$n\widehat\pi_n\to\infty$.  Theorem \ref{thm:PoissonApproximation}\textup{(ii)} gives
$$
\frac{\#\{u\in\mathcal T^n:
\widehat G_{k_n}(\mathcal T_u^n)=1\}}{n\widehat\pi_n}
\longrightarrow1
\qquad\text{in probability}.
$$
The declumped count therefore tends to infinity in probability.
\end{proof}

For every finite tree $T$, let
$$
N_k(T)=\sum_{u\in T}G_k(T_u),
\qquad
\widehat{N}_k(T)=\sum_{u\in T}\widehat{G}_k(T_u)
$$
be respectively the raw and declumped counts. It will be useful to go from declumped counts to raw counts.

\begin{lemma}
\label{lem:declump_to_raw}
Assume that $\P(G_{k_n}(\mathcal T^n)=1)\to0$ and that the following
propagation to the root property holds: for every finite plane tree $T$, every
vertex $v\in T$, every child $u$ of $v$, and every $k$,
$$
G_k(T_u)=1
\quad\textrm{and}\quad
k_v(T)\notin B
\qquad\Longrightarrow\qquad
G_k(T_v)=1.
$$
Assume moreover that for a sequence $(\lambda_{n})$, $\dTV(\widehat{N}_{k_n}(\mathcal T^n),\Po(\lambda_n))\to0$.
Then $\P(N_{k_n}(\mathcal T^n)=0)-e^{-\lambda_n}\to0$.
\end{lemma}

\begin{proof}
First observe that, for every finite tree $T$, $\widehat{N}_k(T)>0$ implies
$N_k(T)>0$. Hence $\{N_k(T)=0\}\subseteq\{\widehat{N}_k(T)=0\}$ and
$$
0\le
\P(\widehat{N}_k(\mathcal T^n)=0)-\P(N_k(\mathcal T^n)=0)
=
\P(N_k(\mathcal T^n)>0,\ \widehat{N}_k(\mathcal T^n)=0).
$$
We shall prove the deterministic inclusion
$$
\{N_k(T)>0,\ \widehat{N}_k(T)=0\}\subseteq\{G_k(T)=1\}.
$$

Let $T$ be a finite tree such that $N_k(T)>0$ and $\widehat{N}_k(T)=0$. Choose a vertex
$u\in T$ of minimal depth such that $G_k(T_u)=1$. We claim that
$u=\varnothing$. Indeed, suppose that $u\neq\varnothing$, and let $v$ be the
parent of $u$. Since $v$ has a child $u$ such that $G_k(T_u)=1$, the condition
$\widehat{N}_k(T)=0$ implies $\widehat{G}_k(T_v)=0$. Therefore necessarily $k_v(T)\notin B$.
By the propagation property, $G_k(T_v)=1$, contradicting the minimality of
$u$. Thus $u=\varnothing$, and hence $G_k(T)=1$.

Applying this inclusion with $T=\mathcal T^n$ and $k=k_n$, we get
$$
0\le
\P(\widehat{N}_{k_n}(\mathcal T^n)=0)-\P(N_{k_n}(\mathcal T^n)=0)
\le
\P(G_{k_n}(\mathcal T^n)=1).
$$
The right-hand side tends to $0$ by assumption. Therefore $
\P(N_{k_n}(\mathcal T^n)=0)-\P(\widehat{N}_{k_n}(\mathcal T^n)=0)\to0$.
Moreover,
$$
\left|\P(\widehat{N}_{k_n}(\mathcal T^n)=0)-e^{-\lambda_n}\right|
\le
\dTV(\widehat{N}_{k_n}(\mathcal T^n),\Po(\lambda_n))
\to0.
$$
It follows that $
\P(N_{k_n}(\mathcal T^n)=0)-e^{-\lambda_n}\to0$,
as desired.
\end{proof}

\begin{corollary}
\label{cor:declump_zero}
Assume the hypotheses of Theorem \ref{thm:PoissonApproximation2}, the
propagation property of Lemma \ref{lem:declump_to_raw}, and
$\pi_n=\P(G_{k_n}(\mathcal T)=1)=O(1/n)$. With the notation
$
N_{k_n}(T)=\sum_{u\in T}G_{k_n}(T_u)$
we have
$$
\P(N_{k_n}(\mathcal T^n)=0)-\exp(-c_Bn\pi_n)  \quad \mathop{\longrightarrow}_{n \rightarrow \infty} \quad 0.
$$
\end{corollary}

\begin{proof}
Theorem \ref{thm:PoissonApproximation2} gives the required Poisson
approximation for the declumped count. Corollary
\ref{cor:SM}\textup{(iii)} gives
$\P(G_{k_n}(\mathcal T^n)=1)\to0$, and Lemma
\ref{lem:declump_to_raw} then compares the two zero events.
\end{proof}

For future use, we finally state the following very simple result.

\begin{lemma}
\label{lem:ineqs}
Let $H_n$ be integer-valued and let $N_n(h),L_n(h)$ be nonnegative random
variables such that
$$
\{H_n\ge h\}=\{N_n(h)\ge1\},
\qquad
\{L_n(h)\ge1\}\subseteq\{H_n\ge h\}.
$$
If $\E[N_n(h_n)]\to0$, then $\P(H_n\ge h_n)\to0$. If
$L_n(h_n)\to\infty$ in probability, then $\P(H_n\ge h_n)\to1$.
\end{lemma}

 %==================================================================
\section{Complete $r$-ary non-fringe subtrees}
\label{sec:rary}
%==================================================================

In this section we establish Theorem \ref{thm:max_string} using the Poisson
approximation framework; declumping is needed only for the unary case. Recall
that $T_u$ denotes the fringe subtree of $T$ at $u$, that we write
$T'\prec T$ if $T$ can be obtained from $T'$ by grafting trees onto the leaves
of $T'$, and that $H_{r,k}$ denotes the complete $r$-ary tree of height $k$ (with all leaves at graph distance $k$ from the root).
We assume throughout that $\mu_r>0$ and set
$$
  G_k(T)=\mathbbm{1}_{\{H_{r,k}\prec T\}},
  \qquad
  \widehat{G}_k(T)
  =
  \mathbbm{1}_{\{k_\varnothing(T)\neq r\}}
  \mathbbm{1}_{\{\exists u\in\Gamma_\varnothing(T):G_k(T_u)=1\}}.
$$
We set
$$
  \pi_k=\P(G_k(\mathcal T)=1),
  \qquad
  \widehat{\pi}_k=\P(\widehat{G}_k(\mathcal T)=1).
$$
Observe that
$$
  \pi_k=\mu_1^k
  \quad \textrm{if } r=1,
  \qquad
  \pi_k=\mu_r^{(r^k-1)/(r-1)}
  \quad \textrm{if } r\ge2.
$$
Finally, let
$$
  N_k(T)=\sum_{u\in T}G_k(T_u),
  \qquad
  \widehat{N}_k(T)=\sum_{u\in T}\widehat{G}_k(T_u)
$$
be respectively the raw and declumped counts.

\subsection{A cut-and-graft bound}

Our two applications below require uniform bounds on the probability of two
nested marked occurrences in order to check assumption \ref{hyp:C}. The following cut-and-graft lemma separates the
descendant fringe subtree from the remainder of the tree and reduces this
two-point estimate to a one-point conditional probability.

  Fix $\ell\ge1$. For a finite tree $T$ with $|T|\ge \ell+1$, set $u=u_{\ell+1}(T)$ and
  define $\operatorname{Cut}_\ell(T)$ as the marked tree obtained from $T$ by
  replacing the fringe subtree $T_u$ by a single marked leaf, denoted by
  $\star$. Conversely, if $A$ is such a marked tree and $S$ is a finite plane
  tree, write $A[S]$ for the tree obtained by grafting $S$ at the marked leaf.
  Observe that the decomposition $T=A[S]$ is unique once the marked leaf is
  fixed.  In particular $\operatorname{Cut}_\ell(A[S])=A$, and the fringe subtree of $A[S]$ rooted
at its $(\ell+1)$-st vertex is $S$.
  
  Let
    $$
    \mathscr A_\ell=\{\operatorname{Cut}_\ell(T): |T|\ge \ell+1\}
  $$
  denote  the family of marked trees obtained in this way. See Figure \ref{fig:grafting_binary} for an application of $\operatorname{Cut}_\ell$.

\begin{lemma}[Cut-and-graft bound]
\label{lem:cut_graft}
Let $\mathcal E$ be an event on finite plane trees, and let
$(\mathcal Q_m)_{m\in I}$ be a finite family of events with
$\P(\mathcal T\in\mathcal Q_m)>0$. Assume that there exist a set
$\mathscr R\subseteq\mathscr A_\ell$ and a map
$\tau:\mathscr R\to I$ such that:
\begin{enumerate}
\item[(a)] for every $A\in\mathscr R$ and every
$S'\in\mathcal Q_{\tau(A)}$, one has $A[S']\in\mathcal E$;
\item[(b)] whenever $|T|\ge\ell+1$, $T\in\mathcal E$, and
$T_{u_{\ell+1}(T)}\in\mathcal E$, one has
$
\operatorname{Cut}_\ell(T)\in\mathscr R$
and
$T_{u_{\ell+1}(T)}
\in
\mathcal Q_{\tau(\operatorname{Cut}_\ell(T))}$.
\end{enumerate}
Then
$$\P\bigl(
|\mathcal T|\ge\ell+1,\
\mathcal T\in\mathcal E,\
\mathcal T_{u_{\ell+1}(\mathcal T)}\in\mathcal E
\bigr)
\leq 
\P(\mathcal T\in\mathcal E)
\max_{m\in I}
\P\bigl(
\mathcal T\in\mathcal E
\mid
\mathcal T\in\mathcal Q_m
\bigr).
$$
\end{lemma}

The lemma separates two nested occurrences by cutting at the root of the
second one. Writing the resulting decomposition as $T=A[S]$, the cut tree
$A$ retains the information responsible for the occurrence at the root,
whereas $S$ carries the descendant occurrence. 
The events $\mathcal Q_m$ encode the minimal condition on $S$ needed to
preserve the root occurrence after grafting. For complete $r$-ary subtrees no condition on $S$ is needed for declumped occurences, while for
leaf-height $\mathcal Q_m$ requires root leaf-height at least $m$.

\begin{proof}[Proof of Lemma \ref{lem:cut_graft}]
For $A\in\mathscr A_\ell$, set
$$
w(A)
=
\prod_{v\in A,\ v\neq\star}\mu_{k_v(A)}.
$$
The Bienaym\'e product formula gives $
\P(\mathcal T=A[S])
=
w(A)\P(\mathcal T=S)$.

For $m\in I$, set
$$
\mathscr R_m
=
\{A\in\mathscr R:\tau(A)=m\},
\qquad
W_m
=
\sum_{A\in\mathscr R_m}w(A).
$$
By condition \textup{(b)} and the cut-and-graft factorization, the probability
in the statement is at most
$$
\sum_{m\in I}
W_m
\P(\mathcal T\in\mathcal E\cap\mathcal Q_m).
$$

On the other hand, the families $\mathscr R_m$ are disjoint. By condition
\textup{(a)} and uniqueness of the cut-and-graft decomposition, all trees
$A[S]$ with $A\in\mathscr R_m$ and $S\in\mathcal Q_m$ belong to
$\mathcal E$, and no tree is counted twice. Therefore
$$
\P(\mathcal T\in\mathcal E)
\ge
\sum_{m\in I}
W_m\P(\mathcal T\in\mathcal Q_m).
$$
Consequently,
\begin{align*}
\sum_{m\in I}
W_m\P(\mathcal T\in\mathcal E\cap\mathcal Q_m)
&\le
\max_{m\in I}
\frac{
\P(\mathcal T\in\mathcal E\cap\mathcal Q_m)
}{
\P(\mathcal T\in\mathcal Q_m)
}
\sum_{m\in I}
W_m\P(\mathcal T\in\mathcal Q_m)
\\
&\le
\P(\mathcal T\in\mathcal E)
\max_{m\in I}
\P\bigl(
\mathcal T\in\mathcal E
\mid
\mathcal T\in\mathcal Q_m
\bigr).
\end{align*}
This completes the proof.
\end{proof}

\subsection{Checking \ref{hyp:S} and \ref{hyp:C}}

Let
$$
  s_{r,k}
  =
  \begin{cases}
    k, & r=1,\\
    (r^k-1)/(r-1), & r\ge2,
  \end{cases}
  \qquad
  \ell_{r,k}
  =
  \begin{cases}
    1, & r=1,\\
    r^k, & r\ge2,
  \end{cases}
$$
so that $s_{r,k}$ is the number of internal vertices of $H_{r,k}$ and
$\ell_{r,k}$ is the number of leaves of $H_{r,k}$.

\begin{lemma}
  \label{lem:S_rary}
  Let $(k_n)$ be such that $s_{r,k_n}+\ell_{r,k_n}=n^{o(1)}$. Fix $\eta\in(0,1-1/\alpha)$ and set $M_n=\lfloor n^{1-\eta}\rfloor$. Then:
  \begin{enumerate}
    \item[(i)] $(G_k)$ satisfies Assumption \ref{hyp:S} along $(k_{n})$ with respect to $(M_n)$;
    \item[(ii)] if $n\pi_{k_n}\to\infty$, then $(G_k)$ satisfies Assumption
    \ref{hyp:C} along $(k_{n})$ with respect to $(M_n)$;
    \item[(iii)] $\E[N_{k_n}(\mathcal T^n)]\sim n\pi_{k_n}$.
  \end{enumerate}
\end{lemma}

\begin{proof}
  On the event $G_k(\mathcal T)=1$, the tree $\mathcal T$ is obtained from
  the complete $r$-ary tree $H_{r,k}$ by grafting independent Bienaym\'e
  trees on the $\ell_{r,k}$ leaves of $H_{r,k}$. Thus, under
  $\P_k=\P(\,\cdot\mid G_k(\mathcal T)=1)$, we may take
  $C_k=s_{r,k}$ and $L_k=\ell_{r,k}$, and construct copies $\mathcal T^{(1)},\mathcal T^{(2)},\ldots$ which,
conditionally on $(C_k,L_k)$, are independent and each have the same law
as $\mathcal T$ such that
  $$
    |\mathcal T|
    =
    C_k+\sum_{j=1}^{L_k}|\mathcal T^{(j)}|.
  $$
  Since $C_{k_n}+L_{k_n}=s_{r,k_n}+\ell_{r,k_n}=n^{o(1)}$, the fact that  $(G_k)$ satisfies Assumption \ref{hyp:S} along $(k_{n})$ with respect to $(M_n)$ follows from Lemma \ref{lem:critM}.

We now prove \textup{(ii)}. Assume that $n\pi_{k_n}\to\infty$. Let
$A_n(\ell)$ be the quantities defining Assumption \ref{hyp:C}. We shall
prove that
$$
\sum_{\ell=1}^{M_n}A_n(\ell)=o(n\pi_{k_n}^2).
$$
  By conditioning on $G_{k_n}(\mathcal T)=1$, we have
  $$
    \sum_{\ell=1}^{M_{n}}A_n(\ell)
    =
    \pi_{k_n}
   \E_{k_n}\left[
      \sum_{\ell=1}^{M_{n}}
      \mathbbm{1}_{\{|\mathcal T|\ge \ell+1\}}
      G_{k_n}(\mathcal T_{u_{\ell+1}(\mathcal T)})
    \right].
  $$
  Under $\P_{k_n}$, use the canonical displayed copy of $H_{r,k_n}$
  at the root. We split the possible descendants counted in the last display
  according to whether their roots belong to this displayed skeleton or to one
  of the Bienaym\'e trees grafted onto its leaves.

First consider roots belonging to the displayed skeleton. If $r=1$, the
vertices of the displayed unary chain occur consecutively in depth-first
order. Hence, whenever $u_{\ell+1}(\mathcal T)$ belongs to this chain, it is
the vertex at graph distance $\ell$ from the root, with
$1\le \ell\le k_n$. The displayed chain already provides a unary
continuation of length $k_n-\ell$ from this vertex. For it to satisfy
$G_{k_n}$, the tree grafted at the terminal leaf must therefore provide
$\ell$ additional unary edges. This has probability
$\pi_\ell=\mu_1^\ell$. Consequently, the expected contribution of the
vertices belonging to the displayed chain is at most
$$
\sum_{\ell=1}^{M_n\wedge k_n}\mu_1^\ell
\le
\sum_{\ell\ge1}\mu_1^\ell
\le C.
$$

  If $r\ge2$, consider a vertex $v$ at graph distance
  $a\in\{1,\ldots,k_n-1\}$ from the root in the displayed copy of $H_{r,k_n}$.
  The displayed skeleton already provides a complete $r$-ary continuation of
  height $k_n-a$ from $v$, with $r^{k_n-a}$ terminal leaves. For $v$ to satisfy
  $G_{k_n}$, each of these terminal leaves must support an additional complete
  $r$-ary tree of height $a$. Since the grafted subtrees are independent, this
  has probability
  $$
    \pi_a^{r^{k_n-a}}
    =
    \frac{\pi_{k_n}}{\pi_{k_n-a}}.
  $$
  The leaves of the displayed skeleton contribute $r^{k_n}\pi_{k_n}$. Hence
  the expected number of counted descendants whose roots belong to the
  displayed skeleton is bounded by
  $$
    \sum_{a=1}^{k_n-1}r^a\,\frac{\pi_{k_n}}{\pi_{k_n-a}}
    +
    r^{k_n}\pi_{k_n}.
  $$
  Since $\pi_k=\mu_r^{(r^k-1)/(r-1)}$ and
  $(r^{k_n}-r^{k_n-a})/(r-1)\ge r^{k_n-1}$ for
  $1\le a\le k_n-1$, this last quantity is bounded by
  $$
    k_n r^{k_n}\mu_r^{r^{k_n-1}}+r^{k_n}\pi_{k_n},
  $$
  and is therefore uniformly bounded in $n$.

  It remains to consider roots lying in the forest grafted on the leaves of the
  displayed skeleton. Embed this forest in an infinite sequence of independent
  Bienaym\'e trees. For each deterministic DFS position in this infinite
  forest, the event that the corresponding fringe subtree exists in the
  grafted forest and has $G_{k_n}$-value $1$ is contained in an event of
  probability $\pi_{k_n}$. Thus a union bound over the first $M_{n}$ possible
  positions gives an expected contribution at most $M_{n}\pi_{k_n}$.

  Consequently,
  $$
    \sum_{\ell=1}^{M_{n}}A_n(\ell)
    \le
    C\pi_{k_n}+M_{n}\pi_{k_n}^2.
  $$
Since $n\pi_{k_n}\to\infty$ and $M_n=o(n)$, the right-hand side is
$o(n\pi_{k_n}^2)$. This proves Assumption \ref{hyp:C} with respect to
$(M_n)$.

Finally, \textup{(iii)} follows from Corollary
\ref{cor:SM}\textup{(ii)}, applied with the same cutoff sequence $(M_n)$.
\end{proof}

We next check Assumption \ref{hyp:C} for the declumped functions
$\widehat{G}_k$.

\begin{lemma}
  \label{lem:C_rary}
  For every sequence $(k_n)$, the functions $(\widehat{G}_k)$ satisfy
  Assumption \ref{hyp:C} with respect to every cutoff sequence $(M_n)$ satisfying $M_n=o(n)$.
\end{lemma}

\begin{proof}
Fix $k,\ell$ and use Lemma \ref{lem:cut_graft} with
$\mathcal E=\{T:\widehat G_k(T)=1\}$ and $\mathcal Q_0$ equal to the whole
tree space.  Let $\mathscr R$ be the  set of all cut trees $A$ such that
$A[S']\in\mathcal E$ for every tree $S'$, and set $\tau(A)=0$.  On the
double event, put $u=u_{\ell+1}(\mathcal T)$ and choose
a root child $c$ with $H_{r,k}\prec\mathcal T_c$.  If
$u\notin\mathcal T_c$, re-grafting at $u$ does not affect this witness.  If
$u\in\mathcal T_c$, then $k_u(\mathcal T)\ne r$ because
$\widehat G_k(\mathcal T_u)=1$; hence $u$ is not an internal vertex of the
displayed $H_{r,k}$.  It is therefore a leaf of that copy or lies below one,
and arbitrary re-grafting at $u$ again preserves the root witness.  Thus the
cut tree belongs to $\mathscr R$, so Lemma \ref{lem:cut_graft}
gives, uniformly in $\ell$,
$$
\P\bigl(|\mathcal T|\ge\ell+1,\ \widehat G_k(\mathcal T)=1,\
\widehat G_k(\mathcal T_{u_{\ell+1}(\mathcal T)})=1\bigr)
\le\widehat\pi_k^2.
$$
Thus
  $$ \sum_{\ell=1}^{M_n}  \P\bigl(|\mathcal T|\ge\ell+1,\ \widehat G_{k_{n}}(\mathcal T)=1,\
\widehat G_{k_{n}}(\mathcal T_{u_{\ell+1}(\mathcal T)})=1\bigr)
\le M_{n}\widehat\pi_{k_{n}}^2= o \left(n \widehat \pi_{k_{n}}^2 \right)$$
since $M_n=o(n)$, which is Assumption \ref{hyp:C}.
\end{proof}

\begin{figure}[h]
    \centering
    \includegraphics[width=0.8\linewidth]{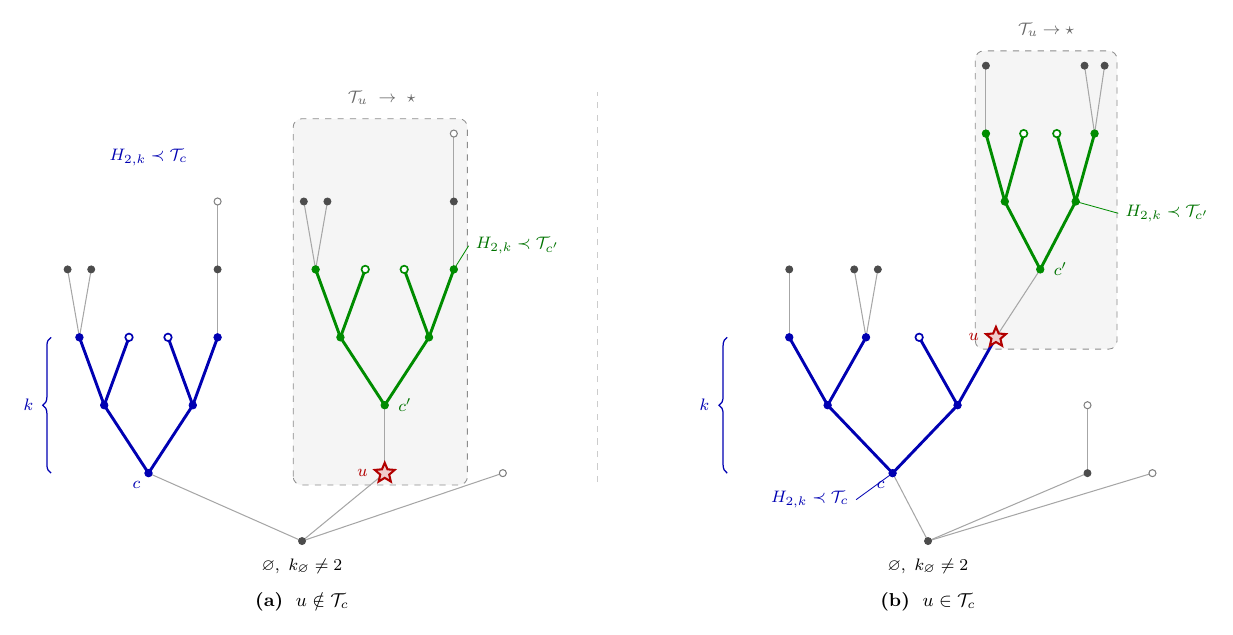}
\caption{The two cases in the proof of Lemma~\ref{lem:C_rary}, in the case $r = 2$. Both the root and $u=u_{\ell+1}(\mathcal T)$ are
declumped occurrences: each has outdegree $\neq 2$ with a child ($c$, resp.
$c'$) whose fringe subtree supports a non-fringe copy of the complete binary tree of height at least $k$. $\operatorname{Cut}_\ell$ replaces $\mathcal T_u$
(dashed box) by a marked leaf $\star$. \emph{(a)} $u$ does not lie in the fringe subtree of $c$, so cutting $\mathcal T_u$ leaves the root witness untouched. \emph{(b)} $u$ lies in the fringe subtree of $c$; since $\deg(u)\neq r$, $u$ cannot be an internal vertex of the copy of the $r$-ary tree. Hence, $u$ lies at distance at least $k$ from $c$ and replacing $\T_u$ with an arbitrary finite plane tree preserves the property $H_{r,k}\prec \T_c$.} 
\label{fig:grafting_binary}
\end{figure}

\subsection{Proof of Theorem \ref{thm:max_string}}

\begin{proof}[Proof of Theorem \ref{thm:max_string}]
Fix $\eta\in(0,1-1/\alpha)$ and set
$M_n=\lfloor n^{1-\eta}\rfloor$. All applications of Assumptions
\ref{hyp:S}, \ref{hyp:M} and \ref{hyp:C}  below are made with
respect to this cutoff sequence.

  We first treat the case $r=1$. Let $t\in\mathbb R$ and set
  $m_n=\lfloor t_n^*+t\rfloor$. Since
  $t_n^*=(\log n+\log(1-\mu_1))/\log(1/\mu_1)$, we have
  $$
    n(1-\mu_1)\pi_{m_n+1}
    =
    n(1-\mu_1)\mu_1^{m_n+1}
    =
    \mu_1^{t+1-\{t_n^*+t\}}.
  $$
Set $h_n=m_n+1$. Since $h_n=O(\log n)$, Lemmas
\ref{lem:S_rary}\textup{(i)} and \ref{lem:C_rary} verify respectively
Assumptions \ref{hyp:S} for $G_k$ and \ref{hyp:C} for $\widehat G_k$.
Take $B=\mathbb Z_+\setminus\{1\}$, so $c_B=1-\mu_1$. The propagation
condition holds because a parent degree outside $B$ is unary and therefore
extends a unary witness from its child. Since $\pi_{h_n}=O(1/n)$,
Corollary \ref{cor:declump_zero} gives
$$
\P(N_{m_n+1}(\mathcal T^n)=0)
-\exp\left(-\mu_1^{t+1-\{t_n^*+t\}}\right)\longrightarrow0.
$$
The equivalence
$\{H_1(\mathcal T^n)\le m_n\}=\{N_{m_n+1}(\mathcal T^n)=0\}$
proves \textup{(i)}.

  We now assume that $r\ge2$. Set $m_n=\lfloor t_n^{**}\rfloor$ and
  $\theta_n=\{t_n^{**}\}$. Since
  $r^{t_n^{**}}=(r-1)\log n/\log(1/\mu_r)$, for every fixed
  $j\in\mathbb Z$,
  $$
    n\pi_{m_n+j}
    \sim
    \mu_r^{-1/(r-1)}\,n^{1-r^{j-\theta_n}}.
  $$

Let $(h_n)$ satisfy $r^{h_n}=O(\log n)$. Lemma
\ref{lem:S_rary}\textup{(i)} and Corollary
\ref{cor:SM}\textup{(ii)} give
$\E[N_{h_n}(\mathcal T^n)]\sim n\pi_{h_n}$. If
$n\pi_{h_n}\to\infty$, Corollary \ref{cor:SM}\textup{(i)}, Lemma
\ref{lem:S_rary}\textup{(ii)}, and Corollary \ref{cor:Non-fringe Poisson} \textup{(ii)} give
$N_{h_n}(\mathcal T^n)\to\infty$ in probability. Applying Lemma \ref{lem:ineqs} with
$$
H_n=H_r(\mathcal T^n),
\qquad
N_n(h)=L_n(h)=N_h(\mathcal T^n)
$$  yields the two
implications
\begin{equation}
\label{eq:threshold_rary}
n\pi_{h_n}\to0\Longrightarrow
\P(H_r(\mathcal T^n)\ge h_n)\to0,
\qquad
n\pi_{h_n}\to\infty\Longrightarrow
\P(H_r(\mathcal T^n)\ge h_n)\to1.
\end{equation}

 We now prove the two-point concentration. Along the indices for which
  $\theta_n\le1/2$, the preceding identity gives
  $n\pi_{m_n+1}\to0$ and $n\pi_{m_n-1}\to\infty$, uniformly over these
  indices. Moreover, $r^{m_n-1}=O(\log n)$. The two estimates above therefore
  give, by \eqref{eq:threshold_rary}, $
    \P(H_r(\mathcal T^n)\in\{m_n-1,m_n\})\to1$
  along these indices.

  Along the indices for which $\theta_n>1/2$, we similarly have
  $n\pi_{m_n+2}\to0$, $n\pi_{m_n}\to\infty$, and
  $r^{m_n}=O(\log n)$. Hence $
    \P(H_r(\mathcal T^n)\in\{m_n,m_n+1\})\to1$
  along these indices. Combining the two sets of indices yields
  $\P(H_r(\mathcal T^n)\in J_n)\to1$.

  Finally, assume that
  $\liminf_n\min(\theta_n,1-\theta_n)>0$. Then there exists
  $\delta>0$ such that $\delta\le\theta_n\le1-\delta$ for all sufficiently
  large $n$. Consequently, $n\pi_{m_n+1}\to0$ and
  $n\pi_{m_n}\to\infty$, while $r^{m_n}=O(\log n)$. Applying again the two
  implications in \eqref{eq:threshold_rary} gives
  $\P(H_r(\mathcal T^n)\ge m_n+1)\to0$ and
  $\P(H_r(\mathcal T^n)\ge m_n)\to1$. Therefore $
    \P(H_r(\mathcal T^n)=m_n)\to1$.

  This completes the proof of Theorem \ref{thm:max_string}.
\end{proof}

%==================================================================
\section{Leaf-height of large trees}
\label{sec:leaf_height_counts}
%==================================================================

In this section we establish Theorem \ref{thm:max_lh}. We use the declumped
Poisson approximation of Theorem \ref{thm:PoissonApproximation2} in the
non-degenerate cases, and the complete $\kappa$-ary result of Theorem
\ref{thm:max_string} in the degenerate case (where $\mu$ is supported on only two integers). Recall that, given a plane tree
$T$, the leaf-height $\lambda_v(T)$ of a vertex $v\in T$ is the graph distance, in number of edges, from $v$ to its closest leaf descendant; in particular, a leaf has leaf-height $0$. We also recall that $\lambda(T)=\max_{v\in T}\lambda_v(T)$.

Set $\kappa=\min\{i\ge1:\mu_i>0\}$ and take
$$
  G_k(T)=\mathbbm{1}_{\{\lambda_\varnothing(T)\ge k\}},
  \qquad
  \widehat{G}_k(T)
  =
  \mathbbm{1}_{\{k_\varnothing(T)\ge \kappa+1\}}
  \mathbbm{1}_{\{\exists u\in\Gamma_\varnothing(T):G_k(T_u)=1\}}.
$$
We set
$$
  \pi_k=\P(G_k(\mathcal T)=1)
  =
  \P(\lambda_\varnothing(\mathcal T)\ge k),
  \qquad
  \widehat{\pi}_k=\P(\widehat{G}_k(\mathcal T)=1),
$$
with the convention $\pi_k=1$ for $k\le0$.

We first establish asymptotics for the leaf-height of the root.

\subsection{Leaf-height asymptotics}
\label{sec:ell_asymp}

\begin{lemma}
  \label{lem:ell_asymp}
  \leavevmode
  \begin{enumerate}
    \item[(i)] If $\mu_1>0$, there exists $C_\mu>0$ such that
    $\pi_n\sim C_\mu\mu_1^n$ as $n\to\infty$.
    \item[(ii)] If $\mu_1=0$, there exists $D_\mu\in(0,1)$ such that
    $
      \pi_n
      \sim
      \mu_\kappa^{-1/(\kappa-1)}D_\mu^{\kappa^n}
    $
    as $n\to\infty$.
  \end{enumerate}
\end{lemma}

\begin{proof}
Let $G_\mu(x)=\sum_{j\ge0}\mu_jx^j$, set
$g(x)=G_\mu(x)-\mu_0$, and write $x_n=\pi_n$. The branching
property gives $x_{n+1}=g(x_n)$, and extinction of the critical
non-degenerate process gives $x_n\downarrow0$.

If $\mu_1>0$, then $g(x)=\mu_1x(1+O(x))$. For any
$\rho\in(\mu_1,1)$, eventually $x_{n+1}\le\rho x_n$, whence
$\sum_nx_n<\infty$. Therefore
$$
\frac{x_n}{\mu_1^n}
=\frac{x_N}{\mu_1^N}
\prod_{j=N}^{n-1}\frac{g(x_j)}{\mu_1x_j}
\longrightarrow C_\mu\in(0,\infty),
$$
because the factors are $1+O(x_j)$.

Assume now that $\mu_1=0$ and put $a=\mu_\kappa$. We have
$$
g(x)=ax^\kappa(1+\varepsilon(x)),
\qquad \varepsilon(x)=O(x).
$$
Choose $\gamma\in(a,1)$. Iterating
$x_{n+1}\le\gamma x_n^\kappa$ from a sufficiently large index gives
constants $A>0$ and $\xi\in(0,1)$ such that
\begin{equation}
\label{eq:ell_superexp}
x_n\le A\xi^{\kappa^n}.
\end{equation}
Set $
y_n=a^{1/(\kappa-1)}x_n$ and $v_n=\kappa^{-n}\log y_n.
$
Then
$$
v_{n+1}-v_n
=\kappa^{-(n+1)}\log(1+\varepsilon(x_n)).
$$
By \eqref{eq:ell_superexp}, the series of increments is absolutely
convergent. Write
$v_n\to\log D_\mu$. The same bound shows that $\log D_\mu<0$, so
$D_\mu\in(0,1)$. Finally,
\begin{align*}
\left|\log y_n-\kappa^n\log D_\mu\right|
&\le
C\kappa^n\sum_{j\ge n}\kappa^{-j}\xi^{\kappa^j}
\longrightarrow0.
\end{align*}
Thus $y_n\sim D_\mu^{\kappa^n}$, or equivalently
$
\pi_n
\sim\mu_\kappa^{-1/(\kappa-1)}D_\mu^{\kappa^n}
$. This completes the proof.
\end{proof}

\begin{remark}
  Set $
    \ell_n=\P(\lambda_\varnothing(\mathcal T)=n)
    =
    \pi_n-\pi_{n+1}$. Then 
    $$
  \textrm{if } \mu_{1}>0, \quad  \ell_n
    \sim
    (1-\mu_1)C_\mu\mu_1^n,  \qquad \textrm{if } \mu_{1}=0,  \quad   \ell_n
    \sim
    \pi_n
    \sim
    \mu_\kappa^{-1/(\kappa-1)}D_\mu^{\kappa^n}.
  $$
  Indeed, if $\mu_{1}>0$ we have $\pi_{n+1}/\pi_n\to\mu_1$ and if $\mu_{1}=0$ then  $\kappa\ge2$ and
  $\pi_{n+1}=g(\pi_n)\sim
  \mu_\kappa\pi_n^\kappa$, so
  $\pi_{n+1}/\pi_n\to0$.  This extends \cite[Lemma 9]{DGZ23}, which shows that  $\ell_n= (\mu_{1}+o(1))^{n}$ when $\mu_{1}>0$ and that $\log \ell_{n}=\Theta (\kappa^{n})$ when $\mu_{1}=0$. For our application we need the more precise estimates of Lemma \ref{lem:ell_asymp}.
\end{remark}

\subsection{Checking \ref{hyp:S} and \ref{hyp:C}}

We set
$$
  c
  =
  \E\left[
    k_\varnothing(\mathcal T)
    \mathbbm{1}_{\{k_\varnothing(\mathcal T)\ge\kappa+1\}}
  \right]
  =
  1-\kappa\mu_\kappa.
$$
In this subsection we assume $c>0$. The degenerate case $c=0$ corresponds to
$\mu$ supported on $\{0,\kappa\}$ and is treated by Theorem \ref{thm:max_string}, since in this case a
  vertex has leaf-height at least $k$ if and only if it supports a complete
  $\kappa$-ary subtree of height $k$.

We first verify Assumption \ref{hyp:S} for
$G_k(T)=\mathbbm{1}_{\{\lambda_\varnothing(T)\ge k\}}$. We set
$$
  N_k(T)=\sum_{v\in T}G_k(T_v).
$$

\begin{lemma}
\label{lem:S_lh}
Let $(k_n)$ be such that $k_n=n^{o(1)}$ if $\kappa=1$, and
$\kappa^{k_n}=n^{o(1)}$ if $\kappa\ge2$. Fix
$\eta\in(0,1-1/\alpha)$ and set
$M_n=\lfloor n^{1-\eta}\rfloor$. Then:
\begin{enumerate}
\item[(i)] $(G_k)$ satisfies Assumption \ref{hyp:S} along $(k_n)$ with
respect to $(M_n)$;
\item[(ii)] $\E[N_{k_n}(\mathcal T^n)]\sim n\pi_{k_n}$.
\end{enumerate}
\end{lemma}

\begin{proof}
  Work under $\P_k=\P(\,\cdot\mid \lambda_\varnothing(\mathcal T)\ge k)$
  and write $\E_k$ for the corresponding expectation. On the event
  $\{\lambda_\varnothing(\mathcal T)\ge k\}$, there are no leaves in
  generations $0,\ldots,k-1$. Let $Z_i$ be the number of vertices at
  generation $i$ in the first $k$ generations of $\mathcal T$, and set
  $C_k=\sum_{i=0}^{k-1}Z_i$ and $L_k=Z_k$. Conditionally on the first $k$
  generations, the fringe subtrees rooted at generation $k$ are independent
  Bienaym\'e trees. Hence one may construct copies $\mathcal T^{(1)},\mathcal T^{(2)},\ldots$ which,
conditionally on $(C_k,L_k)$, are independent and each have the same law
as $\mathcal T$, such that
  $$
    |\mathcal T|
    =
    C_k+\sum_{i=1}^{L_k}|\mathcal T^{(i)}|
    \qquad\textrm{under }\P_k.
  $$

In order to apply Lemma \ref{lem:critM}, it remains to estimate
$\E_k[C_k+L_k]$. Recall that
  $\pi_{s+1}=g(\pi_s)$, where $g(x)=G_\mu(x)-\mu_0$. A vertex which has to
  remain protected for another $s\ge1$ generations has offspring distribution
  $$
    \widetilde\mu_s(j)=\frac{\mu_j\pi_{s-1}^j}{\pi_s},
    \qquad j\ge1,
  $$
  and mean
  $$
    m_s
    =
    \sum_{j\ge1}j\widetilde\mu_s(j)
    =
    \frac{\pi_{s-1}g'(\pi_{s-1})}{\pi_s}.
  $$
  If $\kappa=1$, then $g(x)=\mu_1x(1+O(x))$ as $x\downarrow0$, and hence
  $m_s=1+O(\pi_{s-1})$. If $\kappa\ge2$, then
  $g(x)=\mu_\kappa x^\kappa(1+O(x))$, and hence
  $m_s=\kappa+O(\pi_{s-1})$. Since $\sum_s \pi_s<\infty$ by Lemma
  \ref{lem:ell_asymp}, it follows that, uniformly in $0\le i\le k$,
  $$
    \E_k[Z_i]
    =
    m_km_{k-1}\cdots m_{k-i+1}
    \le
    \begin{cases}
      C, & \kappa=1,\\
      C\kappa^i, & \kappa\ge2.
    \end{cases}
  $$
  Therefore
  $$
    \E_k[C_k+L_k]
    \le
    \begin{cases}
      Ck, & \kappa=1,\\
      C\kappa^k, & \kappa\ge2.
    \end{cases}
  $$
The assumption on $(k_n)$ gives
$\E_{k_n}[C_{k_n}+L_{k_n}]=n^{o(1)}$. Lemma \ref{lem:critM}, applied with
the cutoff $M_n=\lfloor n^{1-\eta}\rfloor$ fixed in the statement, therefore
gives Assumption \ref{hyp:S}, proving \textup{(i)}. Part \textup{(ii)}
follows from Corollary \ref{cor:SM}\textup{(ii)}, applied with the same
cutoff sequence.
\end{proof}

We next verify Assumption \ref{hyp:C} for the declumped functions
$\widehat{G}_k$. By the branching property,
$$
  \widehat{\pi}_k
  =
  \E\left[
    \mathbbm{1}_{\{k_\varnothing(\mathcal T)\ge\kappa+1\}}
    \left(1-(1-\pi_k)^{k_\varnothing(\mathcal T)}\right)
  \right].
$$
Since $\pi_k\to0$, dominated convergence gives
$
  \widehat{\pi}_k\sim c\pi_k$  as  $k\to\infty$.

We shall use the following elementary estimate.

\begin{lemma}
  \label{lem:lh_conditional_mark}
  There exists $C>0$ such that, for every $0\le m\le k$,
  $$
    \P(\widehat{G}_k(\mathcal T)=1,\ \lambda_\varnothing(\mathcal T)\ge m)
    \le
    C\pi_k\pi_m.
  $$
\end{lemma}

\begin{proof}
  For $m=0$, the bound follows from
  $\P(\widehat{G}_k(\mathcal T)=1)=\widehat{\pi}_k=O(\pi_k)$. Assume now
  $m\ge1$. Conditionally on $D=k_\varnothing(\mathcal T)=d$, the event
  $\{\lambda_\varnothing(\mathcal T)\ge m\}$ requires all $d$ child subtrees
  to have root leaf-height at least $m-1$, while $\widehat{G}_k(\mathcal T)=1$
  requires $d\ge\kappa+1$ and at least one child subtree to have root
  leaf-height at least $k$. Hence, by a union bound,
  $$
    \P(\widehat{G}_k(\mathcal T)=1,\ \lambda_\varnothing(\mathcal T)\ge m)
    \le
    \pi_k
    \sum_{d\ge\kappa+1}d\mu_d\pi_{m-1}^{d-1}.
  $$
  Since $d-1\ge\kappa$ on the range $d\ge\kappa+1$ and $\mu$ is critical, the
  last display is bounded by
  $$
    \pi_k\pi_{m-1}^{\kappa}\sum_{d\ge0}d\mu_d
    =
    \pi_k\pi_{m-1}^{\kappa}.
  $$
  Finally, $\pi_m=g(\pi_{m-1})\ge\mu_\kappa\pi_{m-1}^{\kappa}$, so the last
  quantity is at most $\mu_\kappa^{-1}\pi_k\pi_m$.
\end{proof}

\begin{lemma}
\label{lem:C_lh}
Assume that $c>0$. For every sequence $(k_n)$, the functions
$(\widehat G_k)$ satisfy Assumption \ref{hyp:C} with respect to every cutoff
sequence $(M_n)$.
\end{lemma}

\begin{proof}
Fix $k,\ell$ and apply Lemma \ref{lem:cut_graft} with
$$
\mathcal E=\{T:\widehat G_k(T)=1\},
\qquad
\mathcal Q_m=\{T:\lambda_\varnothing(T)\ge m\},
\quad 0\le m\le k.
$$
For $A\in\mathscr A_\ell$, let $c_\star$ be the root child above $\star$
and put $h(A)=d_A(c_\star,\star)$.  Let $\mathscr R$ consist of  the  cut trees
with root degree at least $\kappa+1$ for which either
\begin{enumerate}
\item[(a)] a root child other than $c_\star$ has leaf-height at least $k$; or
\item[(b)] every unmarked leaf below $c_\star$ is at distance at least $k$
from $c_\star$.
\end{enumerate}
Set $\tau(A)=0$ in case \textup{(a)}, and
$\tau(A)=(k-h(A))_+$ otherwise.  In case \textup{(a)}, arbitrary re-grafting
preserves the other-child witness.  In case \textup{(b)}, a graft $S'$ with
$\lambda_\varnothing(S')\ge\tau(A)$ also preserves the witness through
$c_\star$, since $h(A)+\tau(A)\ge k$.  This verifies condition
\textup{(a)} of Lemma \ref{lem:cut_graft}.

On the double event write $\mathcal T=A[S]$.  If an unaffected child
witnesses the root event, then $\tau(A)=0$.  Otherwise the witness passes
through $c_\star$; all unmarked leaves below it are at distance at least
$k$, and
$\lambda_\varnothing(S)\ge(k-h(A))_+=\tau(A)$.  Thus condition
\textup{(b)} also holds.  Consequently,
$$
\P\bigl(|\mathcal T|\ge\ell+1,\ \widehat G_k(\mathcal T)=1,\
\widehat G_k(\mathcal T_{u_{\ell+1}(\mathcal T)})=1\bigr) \le\widehat\pi_k\max_{0\le m\le k}
\frac{\P(\widehat G_k(\mathcal T)=1,\
\lambda_\varnothing(\mathcal T)\ge m)}{\pi_m}
\le C\widehat\pi_k\pi_k,
$$
by Lemma \ref{lem:lh_conditional_mark}.  Finally, if
$D=k_\varnothing(\mathcal T)$ and $b=\P(D\ge\kappa+1)>0$, then
$$
\widehat\pi_k
=\E\left[\mathbbm{1}_{\{D\ge\kappa+1\}}
\bigl(1-(1-\pi_k)^D\bigr)\right]\ge b\pi_k.
$$
Hence the preceding probability is at most $C'\widehat\pi_k^2$, uniformly
in $\ell$.  Since $M_{n}=o(n)$, Assumption \ref{hyp:C} follows for
every $(k_n)$ and every cutoff sequence.
\end{proof}

\begin{figure}[h]
    \centering
    \includegraphics[width=0.8\linewidth]{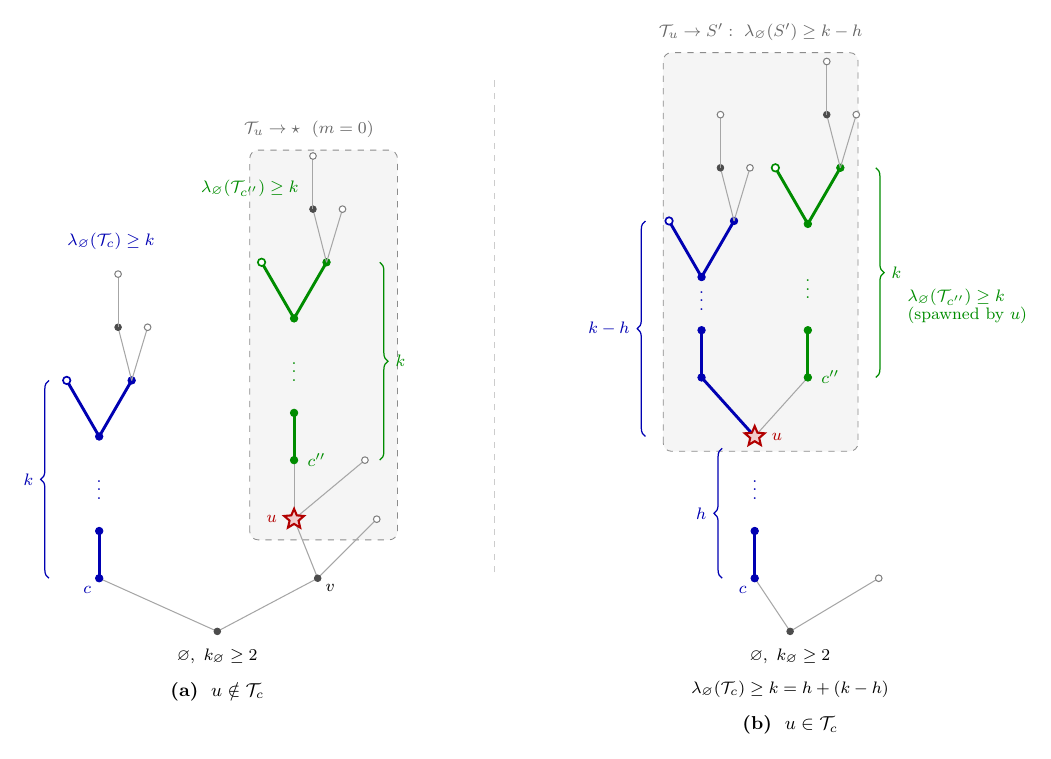}
\caption{The two cases in the proof of Lemma~\ref{lem:C_lh} (leaf-height), in the regime $\mu_1>0$. Both the root and $u=u_{\ell+1}(\mathcal T)$ are
declumped occurrences: each has outdegree $\ge2$ with a child ($c$, resp.
$c''$) whose subtree has root leaf-height at least $k$. A leaf-height witness is not a bare string: its first $k$ generations
contain no leaves and may branch, while arbitrary (grey) trees may be
grafted at generation $k$. $\operatorname{Cut}_\ell$ replaces $\mathcal T_u$
(dashed box) by a marked leaf. \emph{(a)} $u$ does not lie in the fringe subtree of $c$, so cutting $\mathcal T_u$ leaves the root witness untouched, and $m=0$. \emph{(b)} $u$ lies in the fringe subtree of $c$ at distance $h$ from $c$, so the two occurrences can overlap; since $\lambda_\varnothing(\mathcal T_c)\ge k$,
the subtree $\mathcal T_u$ is forced to have $\lambda_\varnothing(\mathcal T_u)\ge
k-h$. Hence only regrafts $S'$ with $\lambda_\varnothing(S')\ge m=(k-h)_{+}$ preserve the
root witness.}
\label{fig:grafting_lh}
\end{figure}

\subsection{Proof of Theorem \ref{thm:max_lh}}

\begin{proof}[Proof of Theorem \ref{thm:max_lh}]
Fix $\eta\in(0,1-1/\alpha)$ and set
$M_n=\lfloor n^{1-\eta}\rfloor$. All applications of Assumptions
\ref{hyp:S}, \ref{hyp:M} and \ref{hyp:C} below are made with respect to
this cutoff sequence.

  For $k\ge0$, write
  $$
    N_k(T)=\sum_{v\in T}G_k(T_v),
    \qquad
    \widehat N_k(T)=\sum_{v\in T}\widehat{G}_k(T_v).
  $$
  We shall use the deterministic implications
  $$
    \widehat N_k(T)\ge1\Rightarrow \lambda(T)\ge k,
    \qquad
    \lambda(T)\ge k\Leftrightarrow N_k(T)\ge1.
  $$

Suppose first that $\mu_1=0$ and $c=1-\kappa\mu_\kappa=0$. By the
definition of $\kappa$ and criticality, $\mu$ is then supported on
$\{0,\kappa\}$ and $\mu_\kappa=1/\kappa$. Every internal vertex has exactly
$\kappa$ children, so deterministically
$
\lambda(T)=H_\kappa(T)$.
Moreover,
$$
\pi_k=\mu_\kappa^{(\kappa^k-1)/(\kappa-1)}
=\mu_\kappa^{-1/(\kappa-1)}
\left(\mu_\kappa^{1/(\kappa-1)}\right)^{\kappa^k}.
$$
Consequently
$D_\mu=\mu_\kappa^{1/(\kappa-1)}$ and
$$
\log(1/D_\mu)=\frac{\log(1/\mu_\kappa)}{\kappa-1}.
$$
Thus $k_n^{**}$ is exactly the centering $t_n^{**}$ of Theorem
\ref{thm:max_string}\textup{(ii)} with $r=\kappa$, which proves
Theorem \ref{thm:max_lh}\textup{(ii)} in this case. It remains below to
consider $c>0$.

  Assume first that $\mu_1>0$. Then $\kappa=1$ and $c=1-\mu_1>0$. By Lemma
  \ref{lem:ell_asymp}, $\pi_k\sim C_\mu\mu_1^k$. Let $t\in\mathbb R$ and set
  $m_n=\lfloor k_n^*+t\rfloor$. Put
  $$
    \Lambda_n=\mu_1^{t+1-\{k_n^*+t\}}.
  $$
  By the definition of $k_n^*$, we have
  $$
    n(1-\mu_1)\pi_{m_n+1}\sim \Lambda_n.
  $$
Set $h_n=m_n+1$. Since $h_n=O(\log n)$, Lemmas
\ref{lem:S_lh}\textup{(i)} and \ref{lem:C_lh} verify the hypotheses of
Theorem \ref{thm:PoissonApproximation2}. Take
$B=\{2,3,\ldots\}$, so $c_B=1-\mu_1$. The propagation condition holds:
a vertex with a marked child and degree outside $B$ must have degree one,
and the leaf-height witness propagates through it. Since
$\pi_{h_n}=O(1/n)$, Corollary \ref{cor:declump_zero} gives
$$
\P(N_{m_n+1}(\mathcal T^n)=0)
-\exp\bigl(-n(1-\mu_1)\pi_{m_n+1}\bigr)\longrightarrow0.
$$
As $n(1-\mu_1)\pi_{m_n+1}-\Lambda_n\to0$ and
$\{\lambda(\mathcal T^n)\le m_n\}=\{N_{m_n+1}(\mathcal T^n)=0\}$, this is
exactly the assertion in \textup{(i)}.

Assume now that $\mu_1=0$, $c>0$, and
$\kappa=\min\{j\ge2:\mu_j>0\}$.
By Lemma
  \ref{lem:ell_asymp},
  $$
    \pi_k\sim \mu_\kappa^{-1/(\kappa-1)}D_\mu^{\kappa^k}.
  $$
  Set $m_n=\lfloor k_n^{**}\rfloor$ and $\theta_n=\{k_n^{**}\}$. Since
  $\kappa^{k_n^{**}}=\log n/\log(1/D_\mu)$, for every fixed
  $j\in\mathbb Z$,
  $$
    n\pi_{m_n+j}
    \sim
    \mu_\kappa^{-1/(\kappa-1)}
    n^{1-\kappa^{j-\theta_n}}.
  $$

For every level sequence below, $\kappa^{h_n}=O(\log n)$. Lemmas
\ref{lem:S_lh}\textup{(i)} and \ref{lem:C_lh}, together with Corollary
\ref{cor:SM}\textup{(ii)}, give
$
\E[N_{h_n}(\mathcal T^n)]\sim n\pi_{h_n}
$.
When $n\pi_{h_n}\to\infty$, Theorem
\ref{thm:PoissonApproximation2} gives
$\widehat N_{h_n}(\mathcal T^n)\to\infty$ in probability.  Applying Lemma \ref{lem:ineqs} with
$$
H_n=\lambda(\mathcal T^n),
\qquad
N_n(h)=N_h(\mathcal T^n),
\qquad
L_n(h)=\widehat N_h(\mathcal T^n).
$$ yields
\begin{equation}
\label{eq:threshold_lh}
n\pi_{h_n}\to0\Longrightarrow
\P(\lambda(\mathcal T^n)\ge h_n)\to0,
\qquad
n\pi_{h_n}\to\infty\Longrightarrow
\P(\lambda(\mathcal T^n)\ge h_n)\to1.
\end{equation}

If $\theta_n\le1/2$, then
$n\pi_{m_n+1}\to0$ and $n\pi_{m_n-1}\to\infty$;
\eqref{eq:threshold_lh} gives
$
\P(\lambda(\mathcal T^n)\in\{m_n-1,m_n\})\to1$.
If $\theta_n>1/2$, then
$n\pi_{m_n+2}\to0$ and $n\pi_{m_n}\to\infty$, and hence
$
\P(\lambda(\mathcal T^n)\in\{m_n,m_n+1\})\to1$.
This proves $\P(\lambda(\mathcal T^n)\in J_n)\to1$.
Finally, if
$\liminf_n\min(\theta_n,1-\theta_n)>0$, then
$n\pi_{m_n+1}\to0$ and $n\pi_{m_n}\to\infty$. A final application of
\eqref{eq:threshold_lh} gives
$
\P(\lambda(\mathcal T^n)=m_n)\to1
$.

  This completes the proof.
\end{proof}

\subsection{Perspectives on vertex leaf-height counts}
\label{ssec:lhperspectives}

We conclude with a possible extension of our results to counts of vertices at
large leaf-height levels.  For fixed $k$, limiting proportions of vertices with
protection number at least $k$ in conditioned Galton--Watson and simply
generated trees were obtained in \cite{DJ14}, and asymptotic
normality of the corresponding counts is covered by the general results for
additive fringe-tree functionals in \cite{Jan16}. See also
\cite{GGLS23} for the protection number of the root and of a
uniformly chosen vertex. Counts at an exact leaf-height, often formulated as
counts of vertices of a given rank, have also been studied in specific random
tree models; see, for example,
\cite{BP17} for random binary search
trees. Recent uniform exponential tail bounds for nearest-leaf distances in
size-conditioned simply generated trees are given in
\cite{MS26}. The bounded-mean
count regime considered below is complementary to  fixed-level and
extremal results.

Recall that $
  \pi_k
  =
  \P(\lambda_\varnothing(\mathcal T)\ge k)$ and that
  $
  \ell_k
  =
  \P(\lambda_\varnothing(\mathcal T)=k)
  =
  \pi_k-\pi_{k+1}$. 
For a sequence $(k_n)$ with $k_n\to\infty$, set
$$
  \Lambda_n^{=}
  =
  \#\{v\in\mathcal T^n:
       \lambda_v(\mathcal T^n)=k_n\},
  \qquad
  \Lambda_n^{\ge}
  =
  \#\{v\in\mathcal T^n:
       \lambda_v(\mathcal T^n)\ge k_n\}.
$$

The distinction between exact and exceedance counts is important. Vertices
with leaf-height at least $k$ may form clusters, notably along unary chains,
whereas a chain of nested exceedances typically contributes only one vertex
with leaf-height exactly $k$. Thus the exact-height count may be viewed as a
naturally declumped version of the exceedance count.

Assume first that $\mu_1>0$. Under the same microscopicity conditions as in
the preceding sections (for instance when $k_n=n^{o(1)}$) and provided
that $(n\ell_{k_n})$ is bounded, we expect
$$
  \dTV\left(
    \Lambda_n^{=},
    \Po(n\ell_{k_n})
  \right)
  \longrightarrow0  \quad\textrm{and}\quad
  \dTV\left(
    \Lambda_n^{\ge},
    \mathrm{CP}(n\ell_{k_n},\gamma_{\mu_1})
  \right)
  \longrightarrow0,
$$
where for $p\in[0,1)$, $\gamma_p$ is the geometric distribution on
$\{1,2,\ldots\}$ defined by $
  \gamma_p(j)
  =
  (1-p)p^{j-1}$ for $j \geq 1$ and
$\mathrm{CP}(\nu,\gamma_p)$ denotes the law of
$\sum_{i=1}^{P}J_i$, where $P$ is Poisson with parameter $\nu$ and,
conditionally on $P$, the variables $J_i$ are independent with law
$\gamma_p$.
The geometric jump distribution reflects the asymptotic length of the unary
chain above an exact-height occurrence. Notice that the mean of the proposed
compound Poisson law is
$$
  \frac{n\ell_{k_n}}{1-\mu_1}
  \sim
  n\pi_{k_n},
$$
as required.

When $\mu_1=0$, we have $\pi_{k+1}/\pi_k\to0$. Hence, in the bounded-mean
regime, occurrences above level $k_n$ should be negligible compared with
occurrences exactly at level $k_n$. We therefore expect
$$
  \dTV\left(
    \Lambda_n^{=},
    \Po(n\ell_{k_n})
  \right)
  \longrightarrow0,
  \qquad
  \dTV\left(
    \Lambda_n^{\ge},
    \Po(n\pi_{k_n})
  \right)
  \longrightarrow0.
$$
Since
$n|\pi_{k_n}-\ell_{k_n}|=n\pi_{k_n+1}\to0$, the second Poisson law may
equivalently be replaced by $\Po(n\ell_{k_n})$.

We expect that these results can be obtained by combining the Poisson
approximation developed in this paper with a marked version of the
declumping argument. We do not pursue these refinements here in order to keep
the length of the paper under control.

\bibliographystyle{alpha}
{\small
\bibliography{bibli.bib}
}

\end{document}